\documentclass[12pt]{amsart}       %was 12pt

\usepackage{amssymb}
\usepackage{wasysym}
\usepackage{bbold}
\usepackage{color}
\usepackage{ifthen}
\usepackage{ifpdf}
\ifpdf
  \usepackage[colorlinks,final,backref=page,hyperindex]{hyperref}
\else
  \usepackage[colorlinks,final,backref=page,hyperindex,hypertex]{hyperref}
\fi
\usepackage{tikz}

\newcommand{\nc}{\newcommand}
\newcommand{\delete}[1]{}

\nc{\mlabel}[1]{\label{#1}}  % Use this to suppress names
\nc{\mcite}[2][]{\ifthenelse{\equal{#1}{}}{\cite{#2}}{\cite[#1]{#2}}}  % Use this to suppress names
\nc{\mref}[1]{\ref{#1}}  % Use this to suppress names
\nc{\mbibitem}[1]{\bibitem{#1}} % Use this to show number name
\delete{
\nc{\mlabel}[1]{\label{#1}  % Use the next two lines to show names
{\hfill \hspace{1cm}{\small\tt{{\ }\hfill(#1)}}}}
\nc{\mcite}[2][]{\ifthenelse{\equal{#1}{}}{\cite{#2}{\small{%
        \tt{{\ }(#2)}}}}{\cite[#1]{#2}{\small{\tt{{\ }(#2)}}}}}  % Use this
\nc{\mref}[1]{\ref{#1}{{\tt{{\ }(#1)}}}}  % Use this lines to show names
\nc{\mbibitem}[1]{\bibitem[\bf #1]{#1}} % Use this to show name
}

\newtheorem{theorem}{Theorem}[section]
\newtheorem{prop}[theorem]{Proposition}
\newtheorem{lemma}[theorem]{Lemma}
\newtheorem{coro}[theorem]{Corollary}
\theoremstyle{definition}
\newtheorem{defn}[theorem]{Definition}

\newtheorem{prop-def}{Proposition-Definition}[section]

\newtheorem{remark}[theorem]{Remark}
\newtheorem{fact}[theorem]{Fact}

\newtheorem{tempex}[theorem]{Example}
\newtheorem{tempexs}[theorem]{Examples}
\newtheorem{temprmk}[theorem]{Remark}
\newtheorem{tempexer}{Exercise}[section]
\newenvironment{exam}{\begin{tempex}\rm}{\end{tempex}}

\nc{\vsa}{\vspace{-.1cm}} \nc{\vsb}{\vspace{-.2cm}}
\nc{\vsc}{\vspace{-.3cm}} \nc{\vsd}{\vspace{-.4cm}}
\nc{\vse}{\vspace{-.5cm}}

\nc{\Irr}{\mathrm{Irr}}
\nc{\ncrbw}{\calr}  %for NC RB words
\nc{\NS}{U_{NS}}
\nc{\FN}{F_{\mathrm Nij}}
\nc{\dfgen}{V} \nc{\dfrel}{R}
\nc{\dfgenb}{\vec{v}} \nc{\dfrelb}{\vec{r}}
\nc{\dfgene}{v} \nc{\dfrele}{r}
\nc{\dfop}{\odot}
\nc{\dfoa}{\dfop^{(1)}} \nc{\dfob}{\dfop^{(2)}}
\nc{\dfoc}{\dfop^{(3)}} \nc{\dfod}{\dfop^{(4)}}
\nc{\mapm}[1]{\lfloor\!|{#1}|\!\rfloor}
\nc{\cmapm}[1]{\frakC(#1)}
\nc{\red}{\mathrm{Red}}
\nc{\cm}{C}
\nc{\supp}{\mathrm{supp}}
\nc{\lex}{\mathrm{lex}}

\nc{\disp}[1]{\displaystyle{#1}}
\nc{\bin}[2]{ (_{\stackrel{\scs{#1}}{\scs{#2}}})}  %binomial coeff
\nc{\binc}[2]{ \left (\!\! \begin{array}{c} \scs{#1}\\
    \scs{#2} \end{array}\!\! \right )}  %binomial coeff
\nc{\bincc}[2]{  \left ( {\scs{#1} \atop
    \vspace{-.5cm}\scs{#2}} \right )}  %binomial coeff
\nc{\sarray}[2]{\begin{array}{c}#1 \vspace{.1cm}\\ \hline
    \vspace{-.35cm} \\ #2 \end{array}}
\nc{\bs}{\bar{S}} \nc{\ep}{\epsilon}
\nc{\dbigcup}{\stackrel{\bullet}{\bigcup}}
\nc{\la}{\longrightarrow} \nc{\cprod}{\ast} \nc{\rar}{\rightarrow}
\nc{\dar}{\downarrow} \nc{\labeq}[1]{\stackrel{#1}{=}}
\nc{\dap}[1]{\downarrow \rlap{$\scriptstyle{#1}$}}
\nc{\uap}[1]{\uparrow \rlap{$\scriptstyle{#1}$}}
\nc{\defeq}{\stackrel{\rm def}{=}} \nc{\dis}[1]{\displaystyle{#1}}
\nc{\sdotcup}{\tiny{ \displaystyle{\bigcup^\bullet}\ }}
\nc{\fe}{\'{e}}
\nc{\hcm}{\ \hat{,}\ } \nc{\hcirc}{\hat{\circ}}
\nc{\hts}{\hat{\shpr}} \nc{\lts}{\stackrel{\leftarrow}{\shpr}}
\nc{\denshpr}{\den{\shpr}}
\nc{\rts}{\stackrel{\rightarrow}{\shpr}} \nc{\lleft}{[}
\nc{\lright}{]} \nc{\uni}[1]{\tilde{#1}} \nc{\free}[1]{\bar{#1}}
\nc{\freea}[1]{\tilde{#1}} \nc{\freev}[1]{\hat{#1}}
\nc{\dt}[1]{\hat{#1}}
\nc{\wor}[1]{\check{#1}}
\nc{\intg}[1]{F_C(#1)}
\nc{\den}[1]{\check{#1}} \nc{\lrpa}{\wr} \nc{\mprod}{\pm}
\nc{\dprod}{\ast_P} \nc{\curlyl}{\left \{ \begin{array}{c} {} \\
{} \end{array}
    \right .  \!\!\!\!\!\!\!}
\nc{\curlyr}{ \!\!\!\!\!\!\!
    \left . \begin{array}{c} {} \\ {} \end{array}
    \right \} }
\nc{\longmid}{\left | \begin{array}{c} {} \\ {} \end{array}
    \right . \!\!\!\!\!\!\!}
\nc{\lin}{\call} \nc{\ot}{\otimes}
\nc{\ora}[1]{\stackrel{#1}{\rar}}
\nc{\ola}[1]{\stackrel{#1}{\la}}%${\Bbb Z}$
\nc{\scs}[1]{\scriptstyle{#1}} \nc{\mrm}[1]{{\rm #1}}
\nc{\margin}[1]{\marginpar{\rm #1}}   %{\rm #1}}
\nc{\dirlim}{\displaystyle{\lim_{\longrightarrow}}\,}
\nc{\invlim}{\displaystyle{\lim_{\longleftarrow}}\,}
\nc{\mvp}{\vspace{0.5cm}}
\nc{\mult}{m}       %multiplication in bialgebra
\nc{\svp}{\vspace{2cm}} \nc{\vp}{\vspace{8cm}}
\nc{\proofbegin}{\noindent{\bf Proof: }}
\nc{\proofend}{$\blacksquare$ \vspace{0.5cm}}
\nc{\sha}{{\mbox{\cyr X}}}  %used to be \cyr
\nc{\ncsha}{{\mbox{\cyr X}^{\mathrm NC}}}
\newfont{\scyr}{wncyr10 scaled 550}
\nc{\ssha}{\mbox{\bf \scyr X}}
\nc{\ncshao}{{\mbox{\cyr X}^{\mathrm NC,\,0}}}
\nc{\shpr}{\diamond}    %Shuffle product
\nc{\shprc}{\shpr_c}
\nc{\shpro}{\diamond^0}    %Shuffle product
\nc{\shpru}{\check{\diamond}} \nc{\spr}{\cdot}
\nc{\catpr}{\diamond_l} \nc{\rcatpr}{\diamond_r}
\nc{\lapr}{\diamond_a} \nc{\lepr}{\diamond_e} \nc{\sprod}{\bullet}
\nc{\un}{u}                 %unit map in bialgebra
\nc{\vep}{\varepsilon} \nc{\labs}{\mid\!} \nc{\rabs}{\!\mid}
\nc{\hsha}{\widehat{\sha}} \nc{\psha}{\sha^{+}} \nc{\tsha}{\tilde{\sha}}
\nc{\lsha}{\stackrel{\leftarrow}{\sha}}
\nc{\rsha}{\stackrel{\rightarrow}{\sha}} \nc{\lc}{\lfloor}
\nc{\rc}{\rfloor} \nc{\sqmon}[1]{\langle #1\rangle}
\nc{\altx}{\Lambda} \nc{\vecT}{\vec{T}} \nc{\piword}{{\mathfrak P}}
\nc{\lbar}[1]{\overline{#1}}
\nc{\dep}{\mathrm{dep}}

\nc{\mmbox}[1]{\mbox{\ #1\ }}
\nc{\ayb}{\mrm{AYB}} \nc{\mayb}{\mrm{mAYB}} \nc{\cyb}{\mrm{cyb}}
\nc{\ann}{\mrm{ann}} \nc{\Aut}{\mrm{Aut}} \nc{\cabqr}{\mrm{CABQR
}} \nc{\can}{\mrm{can}} \nc{\colim}{\mrm{colim}}
\nc{\Cont}{\mrm{Cont}} \nc{\rchar}{\mrm{char}}
\nc{\cok}{\mrm{coker}} \nc{\dtf}{{R-{\rm tf}}} \nc{\dtor}{{R-{\rm
tor}}}

\nc{\Div}{{\mrm Div}} \nc{\End}{\mrm{End}} \nc{\Ext}{\mrm{Ext}}
\nc{\FG}{\mrm{FG}} \nc{\Fil}{\mrm{Fil}} \nc{\Frob}{\mrm{Frob}}
\nc{\Gal}{\mrm{Gal}} \nc{\GL}{\mrm{GL}} \nc{\Hom}{\mrm{Hom}}
\nc{\hsr}{\mrm{H}} \nc{\hpol}{\mrm{HP}} \nc{\id}{\mrm{id}} \nc{\Id}{\mathrm{Id}}
\delete{
\nc{\ID}{\mathrm{ID}}}
\nc{\im}{\mrm{im}} \nc{\incl}{\mrm{incl}} \nc{\Loday}{\mrm{ABQR}\
} \nc{\length}{\mrm{length}} \nc{\LR}{\mrm{LR}} \nc{\mchar}{\rm
char} \nc{\pmchar}{\partial\mchar} \nc{\map}{\mrm{Map}}
\nc{\MS}{\mrm{MS}} \nc{\OS}{\mrm{OS}} \nc{\NC}{\mrm{NC}}
\nc{\rba}{\rm{Rota-Baxter algebra}\xspace}
\nc{\rbas}{\rm{Rota-Baxter algebras}\xspace}
\nc{\rbw}{\rm{RBW}\xspace}
\nc{\rbws}{\rm{RBWs}\xspace}
\nc{\rbadj}{\rm{RB}\xspace}
\nc{\mpart}{\mrm{part}} \nc{\ql}{{\QQ_\ell}} \nc{\qp}{{\QQ_p}}
\nc{\rank}{\mrm{rank}} \nc{\rcot}{\mrm{cot}} \nc{\rdef}{\mrm{def}}
\nc{\rdiv}{{\rm div}} \nc{\rtf}{{\rm tf}} \nc{\rtor}{{\rm tor}}
\nc{\res}{\mrm{res}} \nc{\SL}{\mrm{SL}} \nc{\Spec}{\mrm{Spec}}
\nc{\tor}{\mrm{tor}} \nc{\Tr}{\mrm{Tr}}
\nc{\mtr}{\mrm{tr}}

\nc{\ab}{\mathbf{Ab}} \nc{\Alg}{\mathbf{Alg}}
\nc{\Bax}{\mathbf{CRB}} \nc{\Algo}{\mathbf{Alg}^0}
\nc{\cRB}{\mathbf{CRB}} \nc{\cRBo}{\mathbf{CRB}^0}
\nc{\RBo}{\mathbf{RB}^0} \nc{\BRB}{\mathbf{RB}}
\nc{\Dend}{\mathbf{DD}} \nc{\bfk}{{\bf k}} \nc{\bfone}{{\bf 1}}
\nc{\base}[1]{{a_{#1}}} \nc{\Cat}{\mathbf{Cat}}
 \nc{\DN}{\mathbf{DN}}
\nc{\NA}{\mathbf{NA}}
\nc{\SDN}{\mathbf{SDN}}
\delete{
\nc{\Diff}{\mathbf{Diff}}}
\nc{\gap}{\marginpar{\bf
Incomplete}\noindent{\bf Incomplete!!}
    \svp}
\nc{\FMod}{\mathbf{FMod}} \nc{\Int}{\mathbf{Int}}
\nc{\Mon}{\mathbf{Mon}}
\delete{
\nc{\RB}{\mathbf{RB}}}
\nc{\remarks}{\noindent{\bf Remarks: }}
\nc{\Rep}{\mathbf{Rep}} \nc{\Rings}{\mathbf{Rings}}
\nc{\Sets}{\mathbf{Sets}} \nc{\bfx}{\mathbf{x}}
\nc{\BA}{{\mathbb A}} \nc{\CC}{{\mathbb C}} \nc{\DD}{{\mathbb D}}
\nc{\EE}{{\mathbb E}} \nc{\FF}{{\mathbb F}} \nc{\GG}{{\mathbb G}}
\nc{\HH}{{\mathbb H}} \nc{\LL}{{\mathbb L}} \nc{\NN}{{\mathbb N}}
\nc{\QQ}{{\mathbb Q}} \nc{\RR}{{\mathbb R}} \nc{\TT}{{\mathbb T}}
\nc{\VV}{{\mathbb V}} \nc{\ZZ}{{\mathbb Z}}

\nc{\cala}{{\mathcal A}} \nc{\calb}{{\mathcal B}}
\nc{\calc}{{\mathcal C}}
\nc{\cald}{{\mathcal D}} \nc{\cale}{{\mathcal E}}
\nc{\calf}{{\mathcal F}} \nc{\calg}{{\mathcal G}}
\nc{\calh}{{\mathcal H}} \nc{\cali}{{\mathcal I}}
\nc{\calj}{{\mathcal J}} \nc{\call}{{\mathcal L}}
\nc{\calm}{{\mathcal M}} \nc{\caln}{{\mathcal N}}
\nc{\calo}{{\mathcal O}} \nc{\calp}{{\mathcal P}}
\nc{\calr}{{\mathcal R}} \nc{\cals}{{\mathcal S}} \nc{\calt}{{\mathcal T}}
\nc{\calw}{{\mathcal W}} \nc{\calx}{{\mathcal X}}
\nc{\CA}{\mathcal{A}}

\nc{\frakA}{{\mathfrak A}}
\nc{\fraka}{{\mathfrak a}}
\nc{\frakB}{{\mathfrak B}}
\nc{\frakb}{{\mathfrak b}}
\nc{\frakC}{{\mathfrak C}} \nc{\frakD}{{\mathfrak D}} \nc{\frakE}{{\mathfrak E}}
\nc{\frakd}{{\mathfrak d}}
\nc{\frakF}{{\mathfrak F}}
\nc{\frakg}{{\mathfrak g}}
\nc{\frakm}{{\mathfrak m}}
\nc{\frakM}{{\mathfrak M}}
\nc{\frakMo}{{\mathfrak M}^0}
\nc{\frakO}{{\mathfrak O}}
\nc{\frakP}{{\mathfrak P}}
\nc{\frakp}{{\mathfrak p}}
\nc{\frakS}{{\mathfrak S}}
\nc{\frakSo}{{\mathfrak S}^0}
\nc{\fraks}{{\mathfrak s}}
\nc{\os}{\overline{\fraks}}
\nc{\frakT}{{\mathfrak T}}
\nc{\frakTo}{{\mathfrak T}^0}
\nc{\oT}{\overline{T}}
\nc{\frakX}{{\mathfrak X}}
\nc{\frakXo}{{\mathfrak X}^0}
\nc{\frakx}{{\mathbf x}}
\nc{\frakTx}{\frakT}      %All rooted trees, correspond to \ncsha(X)
\nc{\frakTa}{\frakT^a}        % rooted trees for \ncsha(A)
\nc{\frakTxo}{\frakTx^0}   % rooted trees for \ncshao(X)
\nc{\caltao}{\calt^{a,0}}   % rooted trees for \ncshao(A)
\nc{\ox}{\overline{\frakx}} \nc{\fraky}{{\mathfrak y}}
\nc{\frakz}{{\mathfrak z}} \nc{\oX}{\overline{X}} \font\cyr=wncyr10

\nc{\tred}[1]{\textcolor{red}{#1}} \nc{\tgreen}[1]{\textcolor{green}{#1}}
\nc{\tblue}[1]{\textcolor{blue}{#1}} \nc{\li}[1]{\tred{Li:#1 }}
\nc{\xing}[1]{\tblue{Xing:#1 }} \nc{\markus}[1]{\tgreen{Markus: #1}}

\nc{\dnx}{\Delta_n X} \nc{\dx}{\Delta X} \nc{\dgp}{{\rm deg_{P}}}
\nc{\dgt}{{\rm deg_{T}}} \nc{\dg}{{\rm deg}} \nc{\ida}{ID($A$)} \nc{\tu}{\tilde{u}}
\nc{\fp}{\langle R, Q\rangle } \nc{\drb}{\rm DRB}
\nc{\fy}{f^{\fraka}} \nc{\Mod}{{\bf Mod}} \nc{\uo}{\Omega}
\nc{\op}{\langle \Omega;*\rangle} \nc{\R}{A}
\nc{\tal}{\langle \lc \ \rc_{\Omega};\ast \rangle}
\nc{\opra}{A[\uo]}

\let\epsilon\varepsilon
\let\phi\varphi

\newcommand{\der}{\partial}
\newcommand{\cum}{{\textstyle \varint}}
\newcommand{\vcum}{{\textstyle \varoint}}
\newcommand{\evl}{\text{\scshape\texttt e}}

\newcommand{\galg}{\mathcal{F}}

\newcommand{\ogalg}{\mathcal{G}}
\newcommand{\sig}[1][]{\ifx&#1&\Sigma\else\Sigma^{(#1)}\fi}
\newcommand{\term}[2][]{\ifthenelse{\equal{#1}{}}{\mathcal{T}_{#2}}{\mathcal{T}_{#2}(#1)}}

\newcommand{\fopring}[2]{#1[#2]}

\newcommand{\nonz}[1]{#1^{\times}}

\def\scum{{\setbox0=%
    \hbox{$\textstyle{\scriptscriptstyle\diagup}{\varint}$}%
    \textstyle{\vcenter{\hbox{$\scriptscriptstyle\diagup$}}\kern-.5\wd0}%
    \!\varint}}
\newcommand{\sder}{\eth}

\newcommand{\CAlg}{\mathbf{CAlg}}
\newcommand{\OpAlg}{\mathbf{Alg}(\Omega)}
\newcommand{\COpAlg}{\mathbf{CAlg}(\Omega)}
\newcommand{\Diff}[1][\lambda]{\mathbf{Diff}_{#1}}
\newcommand{\Diffz}{\mathbf{Diff}}
\newcommand{\RB}[1][\lambda]{\mathbf{RB}_{#1}}
\newcommand{\RBz}{\mathbf{RB}}
\newcommand{\DRB}[1][\lambda]{\mathbf{DRB}_{#1}}
\newcommand{\DRBz}{\mathbf{DRB}}
\newcommand{\ID}[1][\lambda]{\mathbf{ID}_{#1}}
\newcommand{\IDz}{\mathbf{ID}}

\newcommand{\ader}{\ell}
\newcommand{\weyl}{\mathrm{A}_1}
\newcommand{\diffweyl}{\weyl(\der)}
\newcommand{\intweyl}{\weyl(\ader)}

\newcommand{\intdiffweyl}{\weyl(\der, \ader)}

\newcommand{\backquote}{^{\,\backprime}}
\newcommand{\grb}{\mathrm{GB}}

\newcommand{\Mid}{\:|\:}
\newcommand{\trl}[2]{[#1]_{#2}}
\newcommand{\sub}[1]{[#1]}
\newcommand{\eopra}{A[\uo|E]}

\begin{document}

\title[Differential Rota-Baxter operators]{On Rings of Differential Rota-Baxter Operators}

\author{Xing Gao}
\address{Department of Mathematics,
    Lanzhou University,
    Lanzhou, Gansu 730000, China}
\email{gaoxing@lzu.edu.cn}
\author{Li Guo}
\address{%
  Department of Mathematics and Computer Science,
  Rutgers University,
  Newark, NJ 07102, USA}
\email{liguo@rutgers.edu}
\author{Markus Rosenkranz}
\address{%
  RISC,
  Johannes Kepler University, A-4040 Linz, Austria}
\email{marcus@rosenkranz.or.at}

\date{\today}

\begin{abstract}
  Using the language of operated algebras, we construct and
  investigate a class of operator rings and enriched modules induced
  by a derivation or Rota-Baxter operator. In applying the general
  framework to univariate polynomials, one is led to the
  integro-differential analogs of the classical Weyl algebra. These
  are analyzed in terms of skew polynomial rings and noncommutative
  Gr{\"o}bner bases.

  \bigskip

  \noindent%
  \textbf{Mathematics Subject Classification:} 16S32; 13N10; 16W99;
  45N05; 47G10; 12H20.

  \medskip

  \noindent%
  \textbf{Keywords:} Differential algebra; Rota-Baxter operators;
  generalized Weyl algebra; skew polynomials; operator rings;
  universal algebra.
\end{abstract}

\maketitle

\tableofcontents

% =============================================================================
\section{Introduction}
\label{sec:intro}
% =============================================================================

The \emph{ring of differential operators}~$\galg[\der]$ over a given
differential ring~$(\galg, \der)$ is a fundamental algebraic
structure\footnote{See the end of this section for conventions on
  notation and terminology.} in the area of differential
algebra~\mcite{Ritt1966,Kolchin1973,CassidyGuoKeigherSit2002},
especially in differential Galois theory~\mcite{PutSinger2003} and
$\mathcal{D}$-module theory~\cite{Coutinho1995}. Building on this
framework and specializing to the case of linear ordinary differential
equations (LODEs), the larger \emph{ring of integro-differential
  operators}~$\galg[\der,\cum] \supset \galg[\der]$ over an
integro-differential ring~$(\galg, \der, \cum)$ was introduced
in~\mcite{Rosenkranz2005,RosenkranzRegensburger2008a} for describing,
computing and factoring the Green's operators of regular boundary
problems for LODEs. As one knows from the classical theory, such
Green's operators will be integrals with the Green's function as its
nucleus\footnote{This is usually called the \emph{kernel}~$k(x,y)$ of
  an integral operator~$f(x) \mapsto \cum k(x,y) \, f(y) \, dy$. In
  algebra, we prefer the less common term \emph{nucleus} for avoiding
  confusion with the kernel of a homomorphism.}. Algebraically
speaking, the Green's operators are contained in the \emph{ring of
  integral operators}~$\galg[\cum] \subset \galg[\der,\cum]$
associated to the Rota-Baxter algebra~$(\galg, \cum)$.

In the present paper we introduce the \emph{ring of differential Rota-Baxter
  operators}~$\galg[\der,\vcum]$ over a given differential Rota-Baxter
algebra~$(\galg, \der, \vcum)$. Although closely related to the integro-differential operator
ring~$\galg[\der,\cum]$, this ring has a more delicate algebraic structure and a distinct range of
applicability. In fact, we shall see that the ring of integro-differential operators is a quotient
of~$\galg[\der,\vcum]$: Loosely speaking, we may view~$(\galg, \der, \vcum)$ as an
integro-differential algebra whose integral is initialized at a ``generic point''; the passage to
the quotient is then interpreted as ``fixing the integration constant''
(Propositions~\ref{prop:lin-op-rings} and~\ref{prop:spec-isom}). For the particular case of
polynomial coefficients $\galg = \bfk[x]$ over a field~$\bfk \supseteq \QQ$, this has been studied
in the context of the (integro-differential) Weyl algebra~\mcite{RegensburgerRosenkranzMiddeke2009},
including the aforementioned \emph{specialization isomorphism} that fixes the integration
constant. For this setting we now provide also a \emph{generalization isomorphism} that goes the
opposite route of embedding the finer structure of differential Rota-Baxter operators into an
integro-differential operator ring containing a generic point (Theorem~\ref{thm:gen-isom-weyl}).

As to be expected from the quotient relation mentioned above, the ring
of differential Rota-Baxter operators~$\galg[\der,\vcum]$ has a
broader range of applicability. In particular, various \emph{classical
  distribution spaces} from analysis can be construed as modules
over~$\galg[\der,\vcum]$ but not over~$\galg[\der,\cum]$,
taking~$\galg = C^\infty(\RR)$ or~$\galg = \RR[x]$ as coefficients
(Example~\ref{ex:vmod}\ref{it:dist}). This is in stark contrast
to~$\galg = C^\infty(\RR)$ or~$\galg = \RR[x]$ itself, which is both
an $\galg[\der,\cum]$-module and an $\galg[\der,\vcum]$-module,
with~$\cum := \vcum := \cum_0^x$ the standard Rota-Baxter operator
on~$\galg$. This is so because distributions can be differentiated
arbitrarily but in general they cannot be evaluated (any point can be
in the singular support of a distribution). In particular, the crucial
identity~$\cum f' = f - f(0)$ for smooth functions~$f \in
C^\infty(\RR)$ fails to hold for distributions~$f \in
\mathcal{D}'(\RR)$. The upshot is that distributions have a sheaf
structure (restrictions to open subsets) but no evaluations
(``restrictions to points'').

Besides the ring of differential Rota-Baxter
operators~$\galg[\der,\vcum]$, which is the main object introduced in
this paper, we have already mentioned the related operator
rings~$\galg[\der]$, $\galg[\cum]$ and~$\galg[\der,\cum]$. It should
be recalled~\mcite[Prop.~17]{RosenkranzRegensburger2008a} that the
latter ring is more than the sum of the two others. In fact
(Proposition~\ref{prop:lin-op-rings}), we have~$\galg[\der,\cum] =
\galg[\der] \dotplus \galg[\cum] \!\setminus\! \galg \dotplus (\evl)$
as $\bfk$-modules, where~$\evl := 1_\galg - \cum \circ \der$ is the
induced evaluation. Having four different operator rings, it will be
expedient to describe a \emph{universal algebraic setting} that allows
to generate these four operator rings---and possibly others---in a
uniform manner (Example~\ref{prop:lin-op-rings}).

In fact, we shall use a slightly more special setting that is better
adapted to our needs: While universal algebra applies to all varieties
(categories whose objects are sets~$A$ endowed with any number of
$n$-operations~$A^n \to A$, subject to laws in equational form), we
shall only need $\bfk$-algebras endowed with one or several unary
operations~$A \to A$, usually known as \emph{operated algebras}. This
leads to significant simplifications: While the algorithmic machinery
of universal algebra is generally dependent on rewriting and the
Knuth-Bendix algorithm, the situation of operated algebras is amenable
to Gr\"obner(-Shirshov)
bases~\mcite{Buchberger1965,Bergman1978,BokutLatyshevShestakovZelmanov2009}. Moreover,
the latter are closely related to and compatible with the skew
polynomial approach used in~\mcite{RegensburgerRosenkranzMiddeke2009}
for constructing the integro-differential Weyl algebra.

It should be emphasized that we allow \emph{arbitrary laws} to be
imposed on operated algebras, not just multilinear laws as one might
be led to expect from the examples. For ground rings of characteristic
zero, we show how to transform arbitrary laws to multilinear ones,
using a suitable polarization process. This may also be the reason why
the usual treatment is based on multilinear laws. For instance
(Example~\ref{ex:lin-var}\ref{it:rba}), Rota-Baxter algebras of weight
zero are normally defined through the axiom~$(\cum f)(\cum g) = \cum f
\cum g + \cum g \cum f$, which may be viewed as the polarized version
of~$(\cum f)^2 = 2 \cum f \cum f$; in characteristic zero these
identities are equivalent. For showing that~$\cum$ is a Rota-Baxter
operator, the latter identity may be better (e.g.\@ using induction on
some degree of~$f$ rather than double induction on~$f$ and~$g$).

Since we have in mind various applications for function algebras, we
restrict ourselves in this paper to operator rings over
\emph{commutative algebras}. However, the construction would work in
essentially the same way for noncommutative algebras, writing the laws
in terms of noncommutative rather than commutative decorated words
(Definition~\ref{def:dec-words}). This could be employed for operator
rings over matrix-valued functions; however, we shall not pursue this
further in the scope of the present paper.

\textbf{Terminology and Notation.} We use~$\NN = \{ 0, 1, 2, \dots \}$
for the \emph{natural numbers with zero}. If~$M$ is a (multiplicative)
monoid~$M$ with zero element~$0 \in M$, the subset of \emph{nonzero
  elements} is denoted by~$\nonz{M} := M \setminus \{0\}$. If~$Z$ is
any set, we write respectively~$M(Z)$ and~$C(Z)$ for the \emph{free
  monoid} and the \emph{free commutative monoid} on~$Z$; its identity
is denoted by~$1$.

Unless specified otherwise, all rings and algebras are assumed to be \emph{associative and unital}
(whereas nonunital rings will be called \emph{rungs}). All modules and algebras are over a
\emph{fixed commutative ring}~$\bfk$, which will be specialized to a field of characteristic zero in
Section~\ref{sec:intdiff-weyl}. All modules are taken to be \emph{left modules}. A (commutative or
noncommutative) ring without zero divisor is called a \emph{domain}. We write~$\Alg_R$ and~$\Mod_R$
for the category of \emph{$R$-modules} and \emph{$R$-algebras}, respectively (suppressing the
subscript~$R$ in the case~$R=\bfk$). We denote the \emph{$\bfk$-span} of a set~$Z$ by~$\bfk Z$; the
ring of \emph{noncommutative polynomials} is thus given by~$\bfk \langle X \rangle := \bfk \, M(X)$
and the ring of \emph{commutative polynomials} by~$\bfk [X] := \bfk \, C(X)$.

Let~$A$ be a~$\bfk$-module with $\bfk$-submodule~$A'$. Then~$A \setminus A'$ denotes a linear
complement of~$A$, rather than the set-theoretic one. (We will only use this notation when such
linear complements exist and the specific choice is irrelevant.)

For~$\der$ and~$\cum$ we employ \emph{operator notation} as in
analysis; for example we write~$(\der f)(\der g)$ rather than~$\der(f)
\der(g)$.  Juxtaposition has precedence over the operators so that we
have for instance~$\der \, f\!g := \der(fg)$ and~$\cum f \cum g :=
\cum (f \cum g)$. Moreover, we use also the customary notation~$f'$
for the \emph{derivative}~$\der f$, and by analogy~$f\backquote$ for
the \emph{antiderivative}~$\cum f$.

\textbf{Structure of the Paper.} In Section~\ref{sec:varieties} we
start by introducing the appropriate tools for describing varieties
and their laws in the framework of $\uo$-operated algebras. The main
result in this section is the reduction of arbitrary laws to
homogeneous and multilinear laws (Corollary~\ref{coro:polar}). We end
the section by introducing the four basic varieties coming from
analysis (Example~\ref{ex:lin-var}). In
Section~\ref{sec:oprings-modules}, the operator ring for a given
variety is introduced (Definition~\ref{def:eopring}). Modules over the
operator rings are described equivalently as a special class of
$\uo$-operated modules (Proposition~\ref{prop:op-module}). The
operator rings and modules are exemplified in the four basic varieties
(Proposition~\ref{prop:lin-op-rings} and
Example~\ref{ex:vmod}). Section~\ref{sec:drbo} is devoted to one of
the four operator rings that is introduced here for the first time:
the ring of differential Rota-Baxter operators. Here the main results
are a left adjoint to the forgetful functor from integro-differential
to differential Rota-Baxter algebras (Theorem~\ref{thm:free-intdiff})
and the embedding of the differential Rota-Baxter operator ring into a
suitable integro-differential operator ring
(Theorem~\ref{thm:gen-isom}). Finally, we turn to the important
special case of polynomial coefficients in
Section~\ref{sec:intdiff-weyl}, thus considering integro-differential
and differential Rota-Baxter analogs for the Weyl algebra. The most
important result is that the so-called integro-differential algebra
introduced in~\cite{RegensburgerRosenkranzMiddeke2009} is in fact the
ring of differential Rota-Baxter operators with polynomial
coefficients (Corollary~\ref{cor:intdiff-isom}), which also implies
the embedding result by specializing the generic one
(Theorem~\ref{thm:gen-isom-weyl}).

% =============================================================================
%\section{Linear Operator Rings}
%\label{sec:lin-op-rings}
% =============================================================================

\section{Varieties of Operated Algebras}
\label{sec:varieties}

Recall that an \emph{$\uo$-operated algebra} $(A; P_\omega \Mid \omega
\in \uo)$ is an algebra $A$ together with certain $\bfk$-linear
operators~$P_\omega\colon A \to A$. Here no restrictions are imposed
on the operators~$P_\omega$. The category of $\uo$-operated algebras
is denoted by~$\OpAlg$, the full subcategory of commutative
$\uo$-operated algebras by~$\COpAlg$. For~$S \subseteq A$, we use the
notation~$(S)$ for the \emph{operated ideal} generated by~$S$.

We describe now the free object~$\frakC_\uo(X)$ of~$\COpAlg$ over a
countable set of generators~$X$. The construction proceeds via
stages~$\frakC_{\uo,n}(X)$ that are defined recursively as follows. We
start with~$\frakC_{\Omega,0}(X):= C(X)$. Then for each~$\omega \in
\Omega$ we create~$\lc C(X) \rc_\omega:=\{\lc u\rc_\omega\,|\, u\in
C(X)\}$ as a disjoint copy of~$C(X)$ and define
$$\frakC_{\Omega,1}(X):= C \big(X\uplus \textstyle\biguplus\limits_{\omega\in \Omega} \lc
C(X)\rc_\omega \big),$$ where~$\uplus$ means disjoint union. Note that
elements in~$\lc C(X)\rc_\omega$ are merely symbols indexed by~$C(X)$;
for example, $\lc 1\rc_\omega$ is not the identity. The
inclusion~$X\hookrightarrow \smash{X\uplus
  \textstyle\biguplus\limits_{\omega\in \Omega}}
\lc\frakC_{\Omega,0}(X) \rc_\omega$ induces a monomorphism
$$i_{0,1}\colon  \frakC_{\Omega,0}(X) = C(X)\hookrightarrow
\frakC_{\Omega,1}(X)=C\big( X\uplus
  \textstyle\biguplus\limits_{\omega\in \Omega} \lc\frakC_{\Omega,0}(X)
\rc_\omega \big)$$
of free commutative monoids through which we
identify~$\frakC_{\Omega,0}(X)$ with its image
in~$\frakC_{\Omega,1}(X)$. For~$n\geq 2$, inductively assume that
$\frakC_{\Omega,n-1}(X)$ has been defined and the embedding
$$i_{n-2,n-1}\colon  \frakC_{\Omega,n-2}(X) \hookrightarrow \frakC_{\Omega,n-1}(X)$$
has been obtained. Then we define
\begin{equation*}
 \label{eq:frakn}
 \frakC_{\Omega,n}(X) := C \big( X\uplus
 \textstyle\biguplus\limits_{\omega \in \Omega} \lc\frakC_{n-1}(X)
 \rc_\omega \big).
\end{equation*}
Since~$\frakC_{\Omega,n-1}(X) = C \big(X\uplus
\smash{\textstyle\biguplus\limits_{\omega \in \Omega}}
\lc\frakC_{\Omega,n-2}(X) \rc_\omega \big)$ is a free commutative
monoid, once again the injections
$$\lc\frakC_{\Omega,n-2}(X) \rc_\omega \hookrightarrow
\lc \frakC_{\Omega,n-1}(X) \rc_\omega$$ induce a monoid embedding
\begin{equation*}
    \frakC_{\Omega,n-1}(X) = C \big( X\uplus
 \textstyle\biguplus\limits_{\omega \in \Omega} \lc\frakC_{n-2}(X)
 \rc_\omega \big)
 \hookrightarrow
    \frakC_{\Omega,n}(X) = C \big( X\uplus
 \textstyle\biguplus\limits_{\omega \in \Omega} \lc\frakC_{n-1}(X)
 \rc_\omega \big).
\end{equation*}

Finally we define the monoid
$$ \frakC_\Omega(X):=\bigcup_{n\geq 0}\frakC_{\Omega,n}(X)$$
whose elements are called (commutative) \emph{$\Omega$-decorated
  bracket words} in~$X$.

\begin{defn}
\mlabel{def:dec-words}
Let $X$ be a set, $\star$ a symbol not in $X$ and $X^\star := X\cup \{\star\}$.
\begin{enumerate}
\item By an \emph{$\Omega$-decorated~$\star$-bracket word on $X$} we
  mean any expression in $\frakC_\Omega(X^\star)$ with exactly one
  occurrence of $\star$. The set of all
  $\Omega$-decorated~$\star$-bracket words on $X$ is denoted by
  $\frakC_\Omega^\star(X)$.
\item For $q\in \frakC_\Omega^\star(X)$ and $u\in \frakC_\Omega(X)$,
  we define $q\sub{u} := q\sub{\star \mapsto u}$ to be the
  $\Omega$-decorated bracket word in $\cmapm{X}$ obtained by replacing
  the letter $\star$ in $q$ by $u$.
\item For $s=\sum_i c_i u_i \in \bfk \frakC_\Omega(X)$, where $c_i\in
  \bfk$, $u_i\in \frakC_\Omega(X)$ and $q\in \frakC_\Omega^\star(X)$,
  we define $q\sub{s} := \sum_i c_i q\sub{u_i},$ which is in
  $\bfk\frakC_\Omega(X)$.
\end{enumerate}
More generally, with $\star_1,\dots\star_n$ distinct symbols not in
$X$, set $X^{\star n} := X\cup \{\star_1, \dots, \star_n \}$.
\begin{enumerate}
\setcounter{enumi}{3}
\item We define an \emph{$\Omega$-decorated~$(\star_1, \dots,
    \star_n)$-bracket word on $X$} to be an expression in
  $\frakC_\Omega (X^{\star n})$ with exactly one occurrence of each of
  $\star_j$, $1\leq j\leq n$. The set of all
  $\Omega$-decorated~$(\star_1, \dots, \star_n)$-bracket words on~$X$
  is denoted by $\frakC_\Omega^{\star n}(X)$.
\item For $q\in \frakC_\Omega^{\star n}(X)$ and $u_1,\dots, u_n\in
  \bfk \frakC_\Omega(X)$, we define
$$q\sub{u_1,\dots, u_n} := q\sub{\star_1 \mapsto u_1, \dots, \star_n \mapsto u_n}$$
to be obtained by replacing the letter $\star_j$ in $q$ by $u_j$ for
$1\leq j\leq n$.
\end{enumerate}
The notation~$q\sub{\theta}$ used above for~$\theta = \{ \star \mapsto
u\}$ and~$\theta = \{ \star_1 \mapsto u_1, \dots, \star_n \mapsto
u_n\}$ can be extended to any substitution~$\theta\colon X^{\star n}
\overset{\sim}{\to} X^{\star n}$; see below
after~Proposition~\ref{pp:freetmn}.
\end{defn}

Now we describe the free object in the category~$\COpAlg$. For
each~$\omega \in \Omega$ we introduce an operator~$\lc\
\rc_\omega\colon \frakC_\Omega(X)\to \frakC_\Omega(X)$ acting
as~$u\mapsto \lc u\rc_\omega$. Then~$(\frakC_\Omega(X); \lc \
\rc_\omega \Mid \omega \in \Omega)$ is a commutative operated monoid;
its linear span~$(\bfk\frakC_\Omega(X); \lc\ \rc_{\omega \in \Omega}
\Mid \omega \in \Omega)$ is a commutative operated algebra. It is
moreover free in the sense of the following
proposition~\mcite{GaoGuoZheng2014,Guo2009}. In the language of
universal algebra, $\bfk\frakC_\Omega(X)$ appears as the term algebra
in the variety of $\uo$-operated algebras~\mcite{BaaderNipkow1998}.

\begin{prop}
  \label{pp:freetm}
  The triple $(\bfk\frakC_\Omega(X); \lc\ \rc_\omega \mid \omega \in
  \uo; j_X)$, with $j_X\colon X \hookrightarrow \frakC_\Omega(X)$ the
  natural embedding, is the free commutative $\uo$-operated algebra
  on~$X$. In other words, for any commutative $\uo$-operated
  algebra~$A$ and any set map~$f \colon X\to A$, there is a unique
  extension of~$f$ to a homomorphism~$\free{f}\colon
  \bfk\frakC_\Omega(X)\to A$ of $\Omega$-operated algebras.
\end{prop}

In the remainder of this section, we assume that~$\bfk$ is a
$\QQ$-algebra. We first define polarization for the non-commutative
case and then induce polarization for the commutative case via a
natural homomorphism. The term \emph{polarization} is adopted from
Rota's early study~\mcite[p.~928]{BillikRota1960} of this normalization process
(second line in the proof of Prop.~2.1).

The construction of noncommutative $\Omega$-decorated bracket
words~$\frakM_\Omega(X)$ is parallel to the commutative
case~$\frakC_\Omega(X)$, using everywhere~$M(X)$ in place of~$C(X)$;
the reader is referred to~\mcite{Guo2009} for details. Clearly,
$\bfk\frakC_\Omega(X)$ is the quotient of $\bfk\frakM_\Omega(X)$
modulo the commutators.

\begin{prop}
  \label{pp:freetmn}
  The triple $(\bfk\frakM_\Omega(X); \lc\ \rc_\omega \mid \omega \in
  \uo; j_X)$, with $j_X\colon X \hookrightarrow \frakM_\Omega(X)$
  again the natural embedding, is the free $\uo$-operated algebra
  on~$X$. This means for any $\uo$-operated algebra~$A$ and for any
  set map~$f \colon X\to A$, there exists a unique extension of~$f$ to
  a homo\-morphism~$\free{f}\colon \bfk\frakM_\Omega(X)\to A$ of
  $\Omega$-operated algebras.
\end{prop}

Operated algebras usually satisfy additional relations, for example
the aforementioned Rota-Baxter axiom~$(\cum f)(\cum g) = \cum f \cum g
+ \cum g \cum f$ in the case of Rota-Baxter algebras. We model such
relations by decorated bracket words~$E \subseteq \bfk
\frakC_\Omega(Y)$ or $E \subseteq \bfk\frakM_\Omega(Y)$, depending on
whether we intend the commutative or noncommutative case. Note that
here we use a new set of variables~$Y$ that should be distinct from
the set~$X$ of generators (see Proposition~\ref{pp:freee} below for an
example combining the two sets of variables). Since relations are
closed under linear combinations, we may take~$E \subseteq \bfk
\frakC_\uo(Y)$ or~$E\subseteq \bfk \frakM_\Omega(Y)$ to
be~$\bfk$-submodules. Elements~$l \in E$ will be called \emph{laws} of
the corresponding variety. In the following, we assume the
noncommutative case but everything can be translated easily to the
commutative case to which we shall return explicitly before
Lemma~\ref{lem:incr}.

For any operated algebra~$\R$ and~$\theta\colon Y\to \R$, by the
universal property of $\bfk\frakM_\Omega(Y)$ as the free
$\Omega$-operated algebra on $Y$, there is a unique morphism of
$\Omega$-operated algebras $\bar{\theta}\colon \bfk\frakM_\Omega(Y)\to
\R$ that extends $\theta$. We use the
notation~$l\sub\theta:=\bar{\theta}(l)$ for the corresponding
\emph{instance} of an~$l \in \bfk \frakM_\uo(Y)$; formally this is the
element of~$\R$ obtained from~$l$ upon replacing every~$y\in Y$
by~$\theta(y)\in \R$, and~$\lc \rc_\omega$ by~$P_\omega$ for~$\omega
\in \uo$. For the special case~$A = \bfk\frakM_\Omega(Y)$ this covers
the substitutions mentioned in Definition~\ref{def:dec-words}.

\begin{defn}
  Let $E$ be a submodule of $\bfk\frakM_\Omega(Y)$.
  \begin{enumerate}
  \item An \emph{$E$-related algebra} is defined to be
    an~$\Omega$-operated algebra~$\R$ such that~$l\sub\theta=0$ for
    any law~$l \in E$ and any assignment~$\theta\colon Y\to \R$.
  \item The \emph{substitution closure}~$\cals(E) \subseteq
    \bfk\frakM_\uo(Y)$ of the laws~$E$ is defined to be the submodule
    spanned by all instances~$l\sub\theta$ with~$l\in E$
    and~$\theta\colon Y\to \bfk\frakM_\Omega(Y)$.
  \end{enumerate}
  If~$E=\{l\}$, we speak of an $l$-algebra, and we write~$\cals(l)$
  for~$\cals(E)$.
\end{defn}

Since it is usually clear from the context that~$E$ denotes a set of
laws (rather than the ground ring), we will often say $E$-algebra
instead of $E$-related algebra. In the terminology of universal
algebra, the category of $E$-algebras (for a fixed set of laws~$E$)
forms a \emph{variety}, which we write here as~$\Alg(\uo|E)$. As
mentioned in the Introduction, the concept of variety is more general
since its operators need not be unary or linear. For our purposes,
this extended generality is not needed and would only complicate
matters, for instance using congruence relations in place of operated
ideals~\mcite[\S1.2]{Cohn2003a}.

\begin{lemma}
  Let $E$ be a submodule of $\bfk\frakM_\Omega(Y)$. Then every
  $E$-algebra is an~$\cals(E)$-algebra, and vice versa.
  \mlabel{lem:subs}
\end{lemma}
\begin{proof}
  The sufficiency is clear since we have~$E \subseteq \cals(E)$. For
  showing the necessity, let~$A$ be an $E$-algebra, and
  take~$l\sub\theta \in \cals(E)$ with~$l\in E$ and~$\theta\colon Y
  \to \bfk \frakM_\uo(Y)$. For any~$\eta\colon Y\to A$,
  define~$\tilde{\eta}\colon Y \to A$ by setting~$\tilde{\eta}(y):=
  \theta(y)\sub{\eta}$ for any~$y\in Y$, that is, $\tilde{\eta}(y)$ is
  obtained by replacing~$y$ in~$\theta(y)$ by~$\eta(y)$. Then we have
  $l\sub\theta\sub{\eta} = l\sub{\tilde{\eta}} = 0$, as~$\R$ is
  an~$E$-algebra and~$l\in E$.
\end{proof}

\begin{exam}
  Let~$\Omega$ be a singleton and~$E = \bfk\{ \lc y_1 y_2\rc - \lc y_1
  \rc y_2 - y_1 \lc y_2\rc \}$. Then
$$\cals(E) = \bfk\{ \lc u v\rc - \lc u\rc v - u \lc v\rc \mid u, v\in
\frakM_\Omega(Y)\}$$ is the substitution closure. (This describes the
variety of differential algebras.)
\end{exam}

Using substitution closure, it is easy to characterize the \emph{free
  $E$-algebra}. Again, this is a special case of a well-known result
in universal algebra~\mcite[Prop.~1.3.6]{Cohn2003a}.

\begin{prop}
  For any submodule~$E \subseteq \bfk\frakM_\uo(Y)$ and any set~$X$,
  let~$S_E=S_E(X)$ denote the operated ideal of~$\bfk\frakM_\uo(X)$
  generated by all~$l\sub\theta$ with~$l\in E$ and~$\theta\colon Y\to
  \bfk\frakM_\uo(X)$. Then the free~$E$-algebra on~$X$ is the quotient
  $F_E(X):=\bfk\frakM_\uo(X)/S_E$.  \mlabel{pp:freee}
\end{prop}

Further exploiting the linear structure of~$\uo$-operated algebras, it
turns out that we may actually assume that~$E$ consists of linear
combinations of multilinear monomials sharing the same variables. Let
us make this precise. Given a monomial~$u\in \frakM_\Omega(Y)$, we
define its \emph{degree in~$y\in Y$}, denoted by $\deg_y u$, as the
number of times that $y$ appears in $u$. Its \emph{total degree} is
given by $\deg u := \sum_{y \in Y} \deg_y u$. Note that if~$\deg_y u =
n$ there exists $q\in \frakM^{\star n}_\Omega(Y\backslash\{y\})$ such
that $u=q\sub{y,\dots,y}$. We call $l$ \emph{multilinear}\footnote{In
  rewriting, this means one can turn the equation~$l = \lambda_1 u_1 +
  \cdots + \lambda_m u_m = 0$ into a \emph{rewrite rule} $u_k \to
  (\lambda_1/\lambda_k) \, u_1 + \cdots + (\lambda_{k-1}/\lambda_k) \,
  u_{k-1} + (\lambda_{k+1}/\lambda_k) \, u_{k+1} + \cdots +
  (\lambda_m/\lambda_k) \, u_m$ for any leading term~$u_k$, and the
  resulting rule will be linear in the sense
  of~\mcite[Def.~6.3.1]{BaaderNipkow1998}. In fact, the rewriting
  terminology allows variables to be absent in terms; this is not
  needed for our present purposes.}  if~$\deg_y l = 1$ for each
variable~$y$ appearing in~$l$.

For any~$l\in E$ and~$y\in Y$, let~$l_{y,n} \: (n \geq 0)$ denote the
linear combination of those monomials of~$l$ that have degree~$n$
in~$y$, with the convention that~$l_{y,n} = \delta_{0,n} l$ if~$y$
does not appear in~$l$. Then~$l$ has the unique \emph{homogeneous
  decomposition}
$$l=\sum\limits_{n\geq 0} l_{y,n}$$ into its $y$-homogeneous parts.

\begin{defn}
  Let~$l\in \bfk \frakM_\Omega(Y)$ be homogeneous in~$y \in Y$
  with~$\deg_y l=n$ such that one has~$l = \sum_{i \le k} c_i \,
  q_i\sub{y, \dots, y}$ with coefficients~$c_i \in \nonz{\bfk}$ and
  monomials~$q_i \in \frakM_\Omega^{\star n}(Y\backslash\{y\})$.  Then
  the \emph{polarization} of~$l$ in~$y$ is
  \begin{displaymath}
    \calp_y(l):= \sum_{\tau\in S_n}\sum_{i \le k} c_i \, q_i\sub{\tau y_1,
      \dots, \tau y_n},    
  \end{displaymath}
  where the \emph{substitution variables}~$y_{1}, \dots, y_{n}\in Y$
  are mutually distinct.
\end{defn}

Note that~$l$ can be recovered (up to a multiple) from~$\calp_y(l)$
through replacing~$y_1,\dots,y_n$ by~$y$; this process is called
\emph{centralization}. For terms containing more variables, we can
also apply polarization so long as the terms are homogeneous in all
variables.

\begin{defn}
  Let~$l\in \bfk\frakM_\Omega(Y)$ be homogeneous in all its
  variables. Then we define its \emph{polarization}~$\calp(l)$ as the
  result of successively polarizing all variables in~$l$.
\end{defn}

A different order of the variables in~$l$ in the polarization process
and a different choice of the substitution variables in~$l$ amounts to
a bijection of the substitution variables. Thus the polarization
of~$l$ is \emph{unique} (up to bijection of variables) and
\emph{multilinear}. Renaming variables if necessary, we may further
assume that the number of variables not appearing in~$E$ is countably
infinite; hence polarization will not run out of substitution
variables.

\begin{exam} \mlabel{exam:pola} In this example, $\uo$ is a singleton
  so that we may abbreviate~$\lc \dots \rc_\omega$ by~$\lc \dots \rc$
  and~$\frakM_\uo(Y)$ by~$\frakM(Y)$.
\begin{enumerate}
\item Consider~$l= \lc y\lc y\rc\rc \in \frakM(Y)$ with~$y\in Y$. Its
  polarization is given by~$\calp(l) = \calp_y(l) = \lc y_1\lc
  y_2\rc\rc + \lc y_2 \lc y_1\rc\rc$, and we recover $2\lc y\lc
  y\rc\rc$ by $y_1, y_2 \mapsto y$.
\item Let~$l= x^2 y^2\in \frakM(Y)$ with~$x,y\in Y$. Then~$\calp_y(l)
  = x^2 y_1y_2 + x^2 y_2y_1$ and hence $\calp(l) = \calp_x(\calp_y(l))
  = y_3y_4y_1y_2 + y_4y_3y_1y_2 + y_3y_4y_2y_1 + y_4y_3y_2y_1$
  with~$y_1, y_2,y_3, y_4\in Y$.
\item For $l=\lc y^2\rc-2\lc y\rc y\in \bfk\frakM(Y)$ we get
  $\calp_y(l)= \lc y_1y_2\rc + \lc y_2y_1 \rc -2\lc y_1\rc y_2 -2\lc
  y_2\rc y_1$.

\end{enumerate}
\end{exam}

\begin{lemma}
  Let~$E\subseteq \bfk\frakM_\Omega(Y)$ be a submodule and~$l\in
  \cals(E)$ arbitrary.  If~$l =\sum_i l_{y,i}$ is the homogeneous
  decomposition of~$l$ in~$y$, we have~$l_{y,i}\in \cals(E)$ for
  each~$i$.
\mlabel{lem:homo}
\end{lemma}

\begin{proof}
  If $n$ is the maximal degree of~$l$ in~$y$, clearly $l_{y,i} = 0 \in
  \cals(E)$ for~$i>n$. Replacing~$y$ in~$l$ by~$jy$ for~$1 \leq j \leq
  n+1$, we obtain
$$ l\sub{jy} = \sum_{i=0}^n j^i \, l_{y,i}.$$
Regard these equations as a linear system in unknowns~$l_{y,i}$. Then
the coefficient matrix is non-singular as a Vandermonde matrix. Thus
one can solve for~$l_{y,0}, \dots, l_{y,n}$ as $\QQ$-linear
combinations of~$l[jy] \in \cals(E)$, which shows that the~$l_{y,i}$
are themselves in~$\cals(E)$.
\end{proof}

We can use the preceding technique to conclude that the polarized form
of a law is always contained in the substitution closure. To see why
this is so, consider a typical example~$l= \lc y\lc y\rc\rc \in
\frakM(Y)$ with~$y\in Y$. Replacing $y$ by~$y_1+y_2$ with~$y_1,y_2\in
Y$, we have
$$\lc(y_1+y_2) \lc y_1+y_2 \rc\rc = \lc y_1 \lc y_1 \rc\rc+ \lc y_2 \lc y_2 \rc\rc + \lc y_1 \lc y_2 \rc\rc + \lc y_2 \lc y_1 \rc\rc \in \cals(l)$$
and so $\calp_y(l):= \lc y_1 \lc y_2 \rc\rc + \lc y_2 \lc y_1 \rc\rc
\in \cals(l)$ by Lemma~\ref{lem:homo}. Let us state the general
result.

\begin{lemma}
  Let~$E\subseteq \bfk\frakM_\Omega(Y)$ be a submodule, and
  assume~$l\in E$ is homogeneous in~$y$. Then we have~$\calp_y(l) \in
  \cals(E)$.
\mlabel{lem:pola}
\end{lemma}

\begin{proof}
  Since $\calp_y$ is a $\bfk$-linear operator, we only need to
  consider~$l \in \frakM_\Omega(Y)$.  We prove the result by induction
  on~$\deg_y l$. If~$\deg_y l = 1$, we have~$\calp_y(l) = l\in
  \cals(E)$. Assuming the result for~$\deg_yl\leq n-1$, we consider
  the case~$\deg_yl = n$. By our assumption on~$l$, we have~$ l =
  q\sub{y, \dots, y}$ with $q \in \frakM_\Omega^{\star
    n}(Y\backslash\{y\})$. Replacing~$y$ by~$z_1+z_2$, we obtain
  \begin{equation}
    \cals(E) \ni q\sub{z_1 + z_2, \cdots, z_1 + z_2}
    =  q\sub{z_1, \dots, z_1} +q\sub{z_2, \dots, z_2} +\tilde{l} .
    \mlabel{eq:polar1}
  \end{equation}
  Since~$q\sub{z_1, \dots, z_1}$ and~$q\sub{z_2, \cdots, z_2}$ are in
  the $\bfk$-module~$\cals(E)$ we have~$\tilde{l} \in \cals(E)$. We
  determine the homogeneous decomposition~$\tilde{l} =
  \sum_{j=1}^{n-1} \tilde{l}_{z_1, j}$ with respect to~$z_1$. From
  Lemma~\mref{lem:homo}, we know~$ \tilde{l}_{z_1, j}\in \cals(E)$. We
  note that $\deg_{z_1} \tilde{l} < n$ and likewise~$\deg_{z_2}
  \tilde{l} < n$. By the induction hypothesis, we
  have~$\calp_{z_1}(\tilde{l}_{z_1, j}) \in \cals(E)$ for~$0 < j <
  n$. Moreover, $\tilde{l}_{z_1, j}$ is homogeneous in~$z_1$ of
  degree~$j > 0$, hence~$\tilde{l}_{z_1, j}$ is homogeneous in~$z_2$
  of degree~$n-j < n$. From the definition
  of~$\calp_{z_1}(\tilde{l}_{z_1, j})$, we see
  that~$\calp_{z_1}(\tilde{l}_{z_1, j})$ is also homogeneous in~$z_2$
  of degree~$n-j < n$. By the induction hypothesis again, we obtain
  now~$\calp_{z_2}(\calp_{z_1}(\tilde{l}_{z_1, j})) \in \cals(E)$.
  Thus it suffices to prove
  \begin{equation}
    \calp_y(l)=\calp_{z_2}(\calp_{z_1}(\tilde{l}_{z_1, j}))
    \mlabel{eq:polar2}
  \end{equation}
  for $0 < j < n$. By its definition, $\calp_y(l)$ is a sum of~$n!$
  terms each of the form~$q\sub{\tau y_1, \dots, \tau y_n}$
  with~$\tau\in S_n$, so we write it as
  \begin{equation}
    \calp_y(l) = \sum_{\tau\in S_n} q\sub{\tau y_1, \dots, \tau y_n}.
    \mlabel{eq:polarell}
  \end{equation}
  On the other hand, we note that
  \begin{equation*}
    q\sub{z_1+z_2,\dots,z_1+z_2}=\sum_{I\subseteq [n]} q\sub{I;z_1,z_2},
    \mlabel{eq:polar3}
  \end{equation*}
  where for each subset $I\subseteq [n]$, the term $q\sub{I;z_1,z_2}$ is
  obtained from $q$ by replacing $\star_i$ by $z_1$ for $i\in I$ and
  by $z_2$ otherwise. For $I=[n]$ and $I=\emptyset$ we obtain the
  first two terms in~(\mref{eq:polar1}), while for $0 < j < n$ we get
  \begin{equation*}
    \tilde{l}_{z_1,j}= \sum_{|I|=j} q\sub{I;z_1,z_2}.    
  \end{equation*}
  Then we have
  \begin{equation}
    \calp_{z_2} (\calp_{z_1}(\tilde{l}_{z_1,j})) =
    \sum_{|I|=j,\tau_1,\tau_2} q\sub{I;\tau_1,\tau_2},
    \mlabel{eq:polarrgh}
  \end{equation}
  where $\tau_1$ ranges over all bijections $I \overset{\sim}{\to}
  [j]$ and $\tau_2$ over all bijections $[n]\setminus I
  \overset{\sim}{\to} [n] \setminus [j]$, and where
  $q\sub{I;\tau_1,\tau_2} := q\sub{\tau y_1,\dots, \tau y_n}$ is
  obtained from $q \in \frakC_\Omega^{\star n}(X)$ via the
  permutation~$\tau \in S_n$ defined by~$\tau(i) = \tau_1(i)$ for~$i
  \in I$ and~$\tau(i) = \tau_2(i)$ for~$i \in [n] \setminus I$. By
  this construction, distinct triples~$(I;\tau_1,\tau_2)$ corresponds to
  distinct permutations $\tau\in S_n$, so distinct monomials in
  Eq.~(\mref{eq:polarrgh}) also correspond to distinct monomials in
  Eq.~(\mref{eq:polarell}). However, there are exactly $\binom{n}{j}
  j! (n-j)!=n!$ such triples, so the sums in the two equations must
  agree, and the proof of~(\mref{eq:polar2}) is complete.
\end{proof}

We introduce now polarization for a collection of laws~$E$.  It turns
out that the resulting module, spanned by multilinear monomials,
defines the same variety as the original~$E$.

\begin{defn}
  Let~$E$ be a submodule of~$\bfk\frakM_\Omega(Y)$. Then we define its
  \emph{polarization}~$\calp(E)$ as the submodule
  of~$\bfk\frakM_\Omega(Y)$ spanned by the polarizations of all
  homogeneous components of elements of~$E$.
\end{defn}

\begin{theorem}
  For any submodule~$E \subseteq \bfk\frakM_\Omega(Y)$,
  an~$\Omega$-operated algebra is an~$E$-algebra if and only if it is
  a~$\calp(E)$-algebra.  \mlabel{thm:polar}
\end{theorem}

\begin{proof}
  By construction, we have~$E\subseteq \calp(E)$. By
  Lemma~\mref{lem:subs}, it suffices to prove that~$\calp(E)$ is
  contained in~$\cals(E)$. Choose a law~$l\in E$ and a variable~$y\in
  Y$ appearing in~$l$. Then~we have~$\calp_y(l)\in \cals(E)$ from
  Lemma~\mref{lem:pola}; repeating the process for the other
  variables com\-pletes~the proof.
\end{proof}

Let us now go back to the \emph{commutative case}. The concepts of
degree and total degree can of course be defined in the same way. By a
straightforward induction on the depth of bracket words, one obtains
the following normalization result.

\begin{lemma}
  Every element of $\frakC_\Omega(Y)$ can be uniquely written as a
  bracket word in which all variables of~$Y$ appear in increasing
  order.  \mlabel{lem:incr}
\end{lemma}

The lemma gives an \emph{embedding}~$\varrho\colon
\bfk\frakC_\Omega(Y)\hookrightarrow \bfk\frakM_\Omega(Y)$ as modules.
On the other hand, as algebras, we have ~$\bfk\frakC_\Omega(Y) \cong
\bfk\frakM_\Omega(Y)/\mathord\sim$, where~$\mathord\sim$ is the
operated ideal of~$\bfk\frakM_\Omega(Y)$ generated by the set~$\{uv
-vu \mid u,v\in \frakM_\Omega(Y)\}$. Let~$\pi\colon
\bfk\frakM_\Omega(Y)\to \bfk\frakC_\Omega(Y)$ be the natural
projection. We carry over the notion of polarization from the
noncommutative case, in the following natural way (by abuse of notation
we continue to use the same symbol~$\calp$ for the commutative
polarization).

\begin{defn}
  Let~$E$ be a submodule of~$\bfk\frakC_\Omega(Y)$.  If~$l\in
  \bfk\frakC_\Omega(Y)$ is homogeneous in all its variables, we define
  its polarization as~$\calp(l):= \pi(\calp(\varrho(l)))$. Similarly,
  the \emph{polarization} of the module is defined as~$\calp(E):=
  \pi(\calp(\varrho(E)))$.
\mlabel{defn:npola}
\end{defn}

\begin{exam} \mlabel{exam:npola}
As in Example~\ref{exam:pola}, we suppress the (unique) operator labels.
\begin{enumerate}
\item Let~$l= x^2 y^2\in \bfk\frakC_\Omega(Y)$ with~$x,y\in Y$. Then
  its polarization is~$\calp(l) = 4y_1y_2y_3y_4$ with~$y_1, y_2,y_3,
  y_4\in Y$.
\item
For $l=\lc y^2\rc-2\lc y\rc y\in \bfk\frakC_\Omega(Y)$ we have $\calp_y(l)=  2\lc y_1y_2\rc -2\lc y_1\rc y_2 -2 y_1\lc y_2\rc$.

\end{enumerate}
\end{exam}

As an immediate corollary to Theorem~\mref{thm:polar}, we obtain that
also in the commutative case one may \emph{polarize all laws} and still
describe the same variety.

\begin{coro}
  For any submodule~$E \subseteq \bfk\frakC_\Omega(Y)$,
  an~$\Omega$-operated algebra is an~$E$-algebra if and only if it is
  a~$\calp(E)$-algebra.
  \mlabel{coro:polar}
\end{coro}

In this sense, it is no loss of generality if one requires that
$E$-algebras be described by \emph{multilinear laws} (but see our
remarks in the Introduction). The classical examples for varieties of
operated algebras are indeed of this form. For avoiding cumbersome
notation, we shall henceforth dispense with the brackets in the main
examples, writing~$\der f$ for~$\lc f \rc_\der$ and~$\cum f$ for~$\lc
f \rc_{{\varint}}$. Likewise, we shall often identify the
operations~$P_\omega\colon A \to A$ of an operated algebra~$(A;
P_\omega \Mid \omega \in \uo)$ with their labels~$\omega$. Note also
that an $E$-algebra is to be understood as the corresponding~$\bfk
E$-algebra if~$E$ is not already a $\bfk$-submodule of~$\bfk
\frakC_\uo(Y)$. Of course, equations of the form~$l = r$ are a
shorthand for~$l-r \in \bfk \frakC_\uo(Y)$.

\begin{exam}
  \label{ex:lin-var}
  Take~$Y = \{ f, g \}$ for the variables. Then the four main
  varieties for doing analysis are the following collections
  of~$E$-algebras with operators~$\uo$.
  \begin{enumerate}
  \item\label{it:diff} The variety~$\Diff$ of
    \emph{differential~$\bfk$-algebras}~\cite{Ritt1966,Kolchin1973,CassidyGuoKeigherSit2002}
    of
    weight~$\lambda \in \bfk$:\\
    Here~$\uo(\Diff) = \{ \der \}$, and~$E(\Diff) := \{ \der \, f\!g =
    (\der f) g + f (\der g) + \lambda \, (\der f)(\der g) \}$ consists
    only of the Leibniz axiom.
  \item\label{it:rba} The variety~$\RB$ of
    \emph{Rota-Baxter~$\bfk$-algebras}~\cite{Baxter1960,Rota1969,Rota1995,Guo2012} of
    weight~$\lambda \in \bfk$:\\
    Here~$\uo(\RB) = \{ \cum \}$ and~$E(\RB) := \{ \smash{(\cum f)
      (\cum g) = \cum f \cum g + \cum g \cum f + \lambda \, \cum f\!g}
    \}$ consists of the Rota-Baxter axiom.
  \item The variety~$\DRB$ of \emph{differential Rota-Baxter
      ~$\bfk$-algebras}~\cite{GuoKeigher2008} of weight~$\lambda \in \bfk$:\\
    Now~$\uo(\DRB) = \uo(\Diff) \cup \uo(\RB) = \{ \der, \cum \}$
    contains both operators, and the laws are given by~$E(\DRB) =
    E(\Diff) \cup E(\RB) \cup \{ \der \cum f = f \}$. The last law is
    the so-called section axiom, which specifies a ``weak coupling''
    between~$\der$ and~$\cum$.
  \item The variety~$\ID$ of \emph{integro-differential
      ~$\bfk$-algebras}~\mcite{GuoRegensburgerRosenkranz2012} of weight~$\lambda \in \bfk$:\\
    This has the same operators~$\uo(\ID) = \uo(\DRB)$ but different
    laws---the weak coupling of~$\DRB$ is replaced by a stronger
    coupling~\mcite[Thm.~2.5]{GuoRegensburgerRosenkranz2012}: From
    various equivalent formulations, we choose
    \begin{equation*}
      E(\ID) = E(\Diff) \,\cup\, \{ f \, \cum g = \cum f' \cum g + \cum
      f\!g + \lambda \, \cum f' \! g, \quad \der \cum f = f \},
    \end{equation*}
    where the middle law describes integration by parts (which is
    strictly stronger than the Rota-Baxter axiom of~$\RB$).
  \end{enumerate}
  As in~\mcite[Def.~8]{RosenkranzRegensburger2008a} we call a
  differential (differential Rota-Baxter, integro-differential)
  algebra \emph{ordinary} if~$\ker{\der} = \bfk$. For
  example,~$(\bfk[x], d/dx)$ is ordinary but~$(\bfk[x,y], \der/\der
  x$) is not.
\end{exam}

The above varieties provide the basic motivation for our study of the
operator rings (to be defined in the next section), which are crucial
for \emph{solving boundary problems} in an algebraic setting. While
this is not the focus of the present paper, the reader may refer to
the end of the next section for some remarks on this application (and
especially on the role played by differential Rota-Baxter algebras).

\section{Operator Rings and Modules}
\label{sec:oprings-modules}

We begin now with the description of the operator rings for a given
variety of operated algebras. This proceeds in \emph{two steps}---we
introduce first a class of operator rings that does not take into
account any law that might be imposed on a given operated algebra
(Definition~\ref{de:freeop}). In the second step we can then impose
the given laws in a suitable form onto the free operators constructed
in the first step (Definition~\ref{def:eopring}). As mentioned in the
Introduction, we work here only with commutative coefficient algebras,
so from now on everything is commutative (except of course the
operator rings).

\begin{defn}
  Let~$(\R; P_\omega \Mid \omega\in \Omega)$ be a commutative
  $\Omega$-operated algebra. Then we define the induced \emph{ring of
    free operators} as the free product $\opra := A*\bfk \langle
  \uo\rangle$.
\label{de:freeop}
\end{defn}

The obvious abuse of notation~$\opra$ is harmless since confusion
with the commutative polynomial ring is unlikely. See also
Remark~\ref{rem:suppress-var} for further justification of this
notation.

If~$A$ is an $E$-algebra for a submodule~$E \subseteq
\bfk\frakC_\Omega(Y)$, we would like an operator ring that reflects
the laws of~$E$; we will construct it as a suitable quotient of the
free operators~$\opra$, using the following \emph{translation from
  laws to operators}. The operator corresponding to a specific law
shall be called the induced relator. For example, a differential
ring~$(\R, \der)$ is an $\uo$-operated algebra with~$\uo = \{ \der \}$
satisfying the Leibniz law~$(fg)' = f' g + f g'$, which induces the
relator~$\der f - f \der - f' \in \opra$; see
Proposition~\ref{prop:lin-op-rings} \ref{it:diffop} for more details.

From Corollary~\mref{coro:polar}, we may assume that~$E \subseteq
\bfk\frakC_\Omega(Y)$ is spanned by homogeneous and multilinear
elements. We may also assume that none of these is of total degree~$0$
since such laws are either redundant (if $l=0$) or else describe a
trivial variety. Since~$Y$ is countable, we can write its elements
as~$y_j \: (j \in \NN)$. Then every basis element of~$E$ having total
degree~$k+1$ can be written in the variables~$y_0, \dots, y_k$ by a
change of variables; the resulting variety remains the same by
Lemma~\mref{lem:subs}. We call such basis elements the \emph{standard
  laws} for the variety. For the translation process, we think of the
lead variable~$y_0$ as the \emph{argument} of the induced relator
with~$y_1, \dots, y_k$ constituting its \emph{parameters}. The latter
can be instantiated by \emph{assignments}, which we view as arbitrary
maps~$\fraka\colon Y' \to A$ on the parameter set~$Y' := Y \setminus
\{ y_0 \}$. Since arguments are processed from right to left, we shall
use the order~$y_k, \dots, y_1, y_0$ in the sequel.

The \emph{induced relator}~$\trl{l}{\fraka} \in \opra$ for a standard
law~$l$ under an assignment~$\fraka$ is now defined by recursion on
the depth of~$l$. Taking~$l \mapsto \trl{l}{\fraka}$ to be
$\bfk$-linear, it suffices to consider monomials~$l$. For the base
case take~$l \in \frakC_{\uo,0}(y_k, \dots, y_1, y_0) = C(y_k, \dots,
y_1, y_0)$ with~$l$ of total degree~$k+1$. By multilinearity~$l = y_k
\cdots y_1 \, y_0$, and we set~$\trl{l}{\fraka} := \fraka(y_k) \cdots
\fraka(y_1)$. Now assume~$\trl{\dots}{\fraka}$ has been defined for
monomial standard laws of depth at most~$n$ and consider $l \in
\frakC_{\uo,n+1}(y_k, \dots, y_1, y_0)$. By multilinearity and the
definition of~$\frakC_{\uo,n+1}$, there exists~$t \in
\frakC_{\uo,n+1}(y_k, \dots, y_1)$ such that either~$l = t y_0$
or~$l = t \lc l' \rc_\omega$ for a certain operator label~$\omega \in
\uo$ and~$l' \in \frakC_{\uo,n}(y_k, \dots, y_1, y_0)$ being a
monomial standard law of depth~$n$. We set~$\trl{l}{\fraka} :=
\bar{\fraka}(t)$ in the former case and use the
recursion~$\trl{l}{\fraka} := \bar{\fraka}(t) \, \omega \,
\trl{l'}{\fraka}$ in the latter, where~$\bar{\fraka}\colon
\frakC_{\uo,n+1}(y_k, \dots, y_1) \to A$ is the monoid homomorphism
induced by the (restricted) assignment map~$\fraka\colon \{ y_k,
\dots, y_1 \} \to A$ through the universal property
of~$\frakC_{\uo,n+1}(y_k, \dots, y_1)$. This completes the definition
of~$\trl{l}{\fraka}$. We can now introduce the ring of $E$-operators
as the quotient of the free operators modulo the translated variety
laws.

\begin{defn}
  Let~$A$ be an $E$-algebra for a submodule~$E \subseteq \bfk\frakC_\Omega(Y)$. Then we define the
  \emph{ring of $E$-operators} as $\eopra := A[\uo] / [E]$, where $[E] \subset \opra$ is the ideal
  generated by~$\trl{l}{\fraka}$ for all standard laws~$l \in E$ and
  assignments~$\mathfrak{a}\colon Y' \to A$.  \mlabel{def:eopring}
\end{defn}

Let us now look at the \emph{classical linear operator rings} for the
varieties of Example~\ref{ex:lin-var}. Each of them comes with a
noncommutative \emph{Gr{\"o}bner basis} and term order, providing
transparent canonical forms and enabling a computational treatment
via the well-known Diamond Lemma~\mcite[Thm.~1.2]{Bergman1978}. As in
the latter reference, we write the elements of the Gr{\"o}bner basis
in the form~$m \to p$ instead of~$m-p$ in order to emphasize the role
of the leading monomial~$m$ and the tail polynomial~$p$, suggesting
their use as rewrite rules.

\begin{remark}
  In the sequel, we identify identities with the varieties they
  define; for example we write~$\galg[\der \Mid \Diff]$
  for~$\galg[\der \Mid E(\Diff)]$, with~$E(\Diff)$ taken from
  Example~\ref{ex:lin-var}\ref{it:diff}. In practice, this notation is
  of course contracted to~$\galg[\der]$, further justifying the abuse
  of notation mentioned after Definition~\mref{de:freeop}.
  \mlabel{rem:suppress-var}
\end{remark}

\begin{prop}
  \label{prop:lin-op-rings}
  Let~$>$ be any graded lexicographic term order
  on~$\fopring{\galg}{\Omega}$ satisfying~$\der > f$ for all~$f \in \galg$
  if~$\der \in \Omega$. Then the following four linear operator
  rings\footnote{Note that we rely on the context to disambiguate the
    notations~$\galg[\der, \vcum]$ and~$\galg[\der, \cum]$. In the
    frame of this paper, a Rota-Baxter operator will always be denoted
    by~$\vcum$ when it comes from a differential Rota-Baxter algebra,
    and by~$\cum$ when it comes from an integro-differential algebra.}
  can be characterized by Gr{\"o}bner bases as follows (primes and
  backprimes refer to the operations in~$\galg$):
  \begin{enumerate}
  \item\label{it:diffop} Given~$(\galg, \der) \in \Diff$, the
    \emph{ring of differential operators}~$\galg[\der] := \galg[\der
    \Mid \Diff]$ has the Gr{\"o}bner basis~$\grb(\Diff) = \{ \der f
    \to f \der + \lambda \, f' \der + f' \mid (f \in \galg) \}$.
  \item\label{it:intop} Given~$(\galg, \cum) \in \RB$, the \emph{ring
      of integral operators}~$\galg[\cum] := \galg[\cum\! \Mid \RB]$
    has the Gr{\"o}bner basis~$\grb(\RB) = \{ \cum f \cum \to
    f\backquote \cum - \cum f\backquote - \lambda \, \cum f \mid f \in
    \galg \}$.
  \item\label{it:diffrbop} Given~$(\galg, \der, \vcum) \in \DRB$, we
    consider next the \emph{ring of differential Rota-Baxter
      operators}~$\galg[\der,\vcum] := \galg[\der,\vcum\! \Mid
    \DRB]$. Its Gr{\"o}bner basis is given by the combined rewrite
    rules~$\grb(\DRB) = \grb(\Diff) \cup \grb(\RB) \cup \{ \der \cum
    \to 1 \}$.
  \item\label{it:intdiffop} For~$(\galg, \der, \cum) \in \ID$, the \emph{ring of
      integro-differential operators}~$\galg[\der,\cum] := \galg[\der,\cum \Mid \ID]$ has
    Gr{\"o}bner
    basis~$\grb(\ID) := \smash{\grb(\DRB) \cup \{ \cum f \der \to \underline{f} - \cum
      \underline{f}' - \evl(\underline{f}) \, \evl \mid f \in \galg \}}$,
    provided\footnote{This is of course always satisfied in the zero weight case (with trivial
      shift). But it is also satisfied in the classical example with weight~$\lambda = \pm 1$: the
      sequence space~$M^\ZZ$ over a~$\bfk$-module~$M$ with forward/backward difference as
      derivation. This has increment/decrement as mutually inverse shifts.}  the
    \emph{shift}~$f \mapsto f + \lambda f'$ has an inverse~$f \mapsto \underline{f}$.
  \end{enumerate}
  We have~$\galg[\der,\cum] = \galg[\der] \dotplus \galg[\cum]
  \!\setminus\! \galg \dotplus (\evl)$ as $\bfk$-modules,
  where~$\galg[\der, \cum]$ contains both~$\galg[\der]$
  and~$\galg[\cum]$ as subalgebras. Moreover, if~$I$ is the ideal
  generated by~$\{ \evl f - \evl(f) \, \evl \mid f \in \galg \}$, we
  have~$\galg[\der,\cum] \cong \galg[\der,\vcum] / I$, where~$\vcum :=
  \cum$ is viewed as part of~$(\galg, \der, \vcum) \in \DRB$.
\end{prop}

\begin{proof}
  Let us first prove the four items stated in the proposition (viewing
  all axioms in the main variable~$g$ and using arbitrary
  assignments~$\fraka$ with~$\fraka(f) \in \galg$ shortened to~$f$):
  \begin{enumerate}
  \item Clearly, the only relators are~$[l]_{\fraka} = \der f - f' - f \der - \lambda \, f' \der$,
    corresponding to the Leibniz
    axiom~$l := \der f\!g - (\der f) g - f (\der g) - \lambda \, (\der f)(\der g) = 0$.
    From~$\der > f$ one sees that the leading monomial is~$\der f$. There is just one S-polynomial
    coming from the overlap ambiguity~$\der f\!g$ between the
    rule~$\der f \to f \der + \lambda \, f' \der + f'$ and the (tacit)
    rule~$f \! g \to f *_\galg g$. Using the Leibniz rule in~$\galg$, one checks immediately that
    the S-polynomial reduces to zero, so~$\grb(\Diff)$ is indeed a Gr{\"o}bner basis.
  \item Here the Rota-Baxter axiom~$l := \smash{(\cum f) (\cum g) -
      \cum f \cum g - \cum g \cum f - \lambda \, \cum f\!g} = 0$
    yields the relators~$[l]_{\fraka} = f\backquote \cum - \cum f \cum -
    \cum f\backquote - \lambda \, \cum f$ whose leading monomial
    is~$\cum f \cum$ because the term order is graded. One obtains an
    S-polynomial from the self-overlap~$\cum f \cum \bar{f} \cum$ of
    the rule~$\cum f \cum \to f\backquote \cum - \cum f\backquote -
    \lambda \, \cum f$. Again one checks that this S-polynomial
    reduces to zero, and~$\grb(\RB)$ is thus a Gr{\"o}bner basis.
  \item The relators are those of~\ref{it:diffop} and~\ref{it:intop},
    and additionally~$\der \cum - 1$ whose corresponding rule is
    clearly~$\der \cum \to 1$ because of the grading. Apart from the
    previous ones, we have the additional overlap ambiguity~$\der \cum
    f \cum$, and again its S-polynomial immediately reduces to zero so
    that~$\grb(\DRB)$ is a Gr{\"o}bner basis.
  \item From the definition~$\evl := 1 - \cum \der$ and the Leibniz rule we have the
    tautological relation
    \begin{equation*}
      fg - \evl(fg) = \cum f'g + \cum fg' + \lambda \, \cum f'g'.
    \end{equation*}
    Let us start by recalling~\mcite[Thm.~2.5(b)]{GuoRegensburgerRosenkranz2012} that the well-known
    integration-by-parts
    axiom~$g \, \cum f - \cum g' \cum f - \cum f\!g - \lambda \, \cum g' \!  f = 0$
    characterizing~$\ID$ is equivalent to the multiplicativity
    condition~$\evl(fg) = \evl(f) \, \evl(g)$. Indeed, by the definition of~$\evl$, the condition
    is~$\cum (fg)' + \cum f' \cdot \cum g' = f \cum g' + g \cum f'$. From this one obtains the axiom
    by expanding~$(fg)'$ according to the (weighted) Leibniz rule and
    substituting~$f \rightsquigarrow \cum f$. Conversely, the axiom of~$\ID$ implies
    that~$\im{\cum}$ is an ideal of~$\galg$. From the definition of~$\evl$ one has the
    identity~$fg = \evl{f} \, \evl{g} + \evl{f} \, \cum g' + \evl{g} \, \cum f' + \cum f' \cdot \cum
    g'$,
    which yields multiplicativity upon applying~$\evl$ since~$\evl$ projects onto~$\ker{\der}$ so
    that~$\evl{f} \, \evl{g} \in \ker{\der}$ is left invariant (note that~$\ker{\der}$ is a
    subalgebra of~$\galg$ because of the Leibniz rule). The three remaining terms are all in the
    ideal~$\cum f$, which is the complement of~$\ker{\der}$ under the projection, hence they vanish
    under~$\evl$.

    We exploit the equivalent characterization of~$\ID$ in terms~$\evl(fg) = \evl(f) \, \evl(g)$ by
    substituting the tautological relation from above to obtain the equivalent law
    \begin{equation*}
      l := \cum fg' - fg + \cum f'g + \lambda \,
      \cum f' g' + \evl(f) \, \evl(g) = 0.
    \end{equation*}
    This axiom gives rise to the new
    relator~$[l]_{\fraka} = \cum (f + \lambda f') \, \der - f + \cum f' + \evl(f) \, \evl$, which
    yields the
    rule~$\cum f \der \to \underline{f} - \cum \underline{f}' - \evl(\underline{f}) \, \evl$ upon
    replacing~$f$ by~$\underline{f}$ and picking~$\cum f \der$ as the leading monomial due to the
    grading. Since the plain Rota-Baxter axiom is implied by the integration-by-parts
    axiom~\mcite[Lem.~2.3(b)]{GuoRegensburgerRosenkranz2012}, the relator constructed in
    Item~\ref{it:diffrbop} is also contained in the current relator ideal, hence the corresponding
    rule is admissible in~$\grb(\ID)$. For seeing that this is again a Gr{\"o}bner basis, we refer
    to the proof of~\mcite[Prop.~13]{RosenkranzRegensburger2008a}. The latter assumes that~$\bfk$ is
    a field and uses a $\bfk$-basis of~$\galg$ but this is only a convenience tuned to the
    algorithmic treatment. As pointed out
    after~\mcite[Prop.~26]{RosenkranzRegensburgerTecBuchberger2012}, choosing a basis is avoided by
    factoring out the linear ideal (this happens in the formation of the free operators~$\opra$ in
    the current setup). Note also that here we take~$\Phi = \{ \evl \}$, which means all rules
    of~\mcite[Table~1]{RosenkranzRegensburger2008a} with characters~$\phi, \psi \in \Phi$ on the
    right-hand side are absent.\footnote{\label{fn:no-eval-rules}The character~$\evl \in \Phi$ is
      not part of the operator set~$\uo$, and its appearance on the right-hand side is to be
      understood merely as an abbreviation~$\evl := 1 - \cum \der$. Moreover, the corresponding
      rules with~$\evl$ on the left-hand side are not required in~$E$ since they follow from the
      other rules.}  With this understanding, the above definition of~$\galg[\der, \cum]$ coincides
    with the one in~\mcite{RosenkranzRegensburger2008a}, which therefore establishes~$\grb(\ID)$ as
    a Gr{\"o}bner basis.
  \end{enumerate}
  We prove now the $\bfk$-module decomposition
  \begin{equation}
    \label{eq:decomp}
    \galg[\der,\cum \Mid \ID] = \galg[\der \Mid \Diff] \dotplus \galg[\cum \Mid \RB]
    \!\setminus\! \galg \dotplus (\evl)
  \end{equation}
  with~$(\evl) \subset \galg[\der,\cum \Mid \ID]$ being the two-sided ideal generated by~$\evl$.
  Note that here and in the rest of this proof, we renounce the abbreviation
  of~$\galg[\Omega \Mid E]$ by~$\galg[\Omega]$ used in the statement of the proposition. This is
  because we need to distinguish the free operator ring from various $E$-operator
  rings. Furthermore, we write~$\galg[\Omega]_E$ for the $\bfk$-submodule of normal forms
  in~$\galg[\Omega]$ with respect to the reduction system induced by~$E$ and the given term order
  on~$\galg[\Omega]$. By the well-known Diamond Lemma~\mcite[Thm.~1.2]{Bergman1978}, we
  have~$\galg[\der, \cum] = \galg[\der, \cum]_{\ID} \dotplus [\ID]$. We claim that it suffices to
  prove
  \begin{equation}
    \label{eq:decomp-norm}
    \galg[\der, \cum]_{\ID} = \galg[\der]_{\Diff} \dotplus \galg[\cum]_{\RB}\!\!\setminus\!\galg
    \dotplus (\evl)_{\ID} .
  \end{equation}
  Indeed, substituting the decomposition~\eqref{eq:decomp-norm} into the Diamond-Lemma decomposition
  and then taking the quotient by~$[\ID]$ yields\footnote{Note that
    $M = A \dotplus B \dotplus C \dotplus Z$ implies
    $M/Z = (A+Z)/Z \dotplus (B+Z)/Z \dotplus (C+Z)/Z$ for arbitrary submodules~$A,B,C,Z$ of some
    module~$M$.}
  \begin{equation}
    \label{eq:decomp-norm-again}
    \galg[\der, \cum \Mid \ID] = \frac{\galg[\der]_{\Diff}+[\ID]}{[\ID]} \dotplus
    \frac{\galg[\cum]_{\RB}\!\!\setminus\!\galg+[\ID]}{[\ID]} \dotplus
    \frac{(\evl)_{\ID}+[\ID]}{[\ID]} .
  \end{equation}
  Since~$\Diff \subset \ID$, we may replace the first denominator on the right-hand side
  of~\eqref{eq:decomp-norm-again}
  by~$\big( \galg[\der]_{\Diff}+[\Diff] \big) + [\ID] = \galg[\der] + [\ID]$, using now the Diamond
  Lemma for~$\Diff$. In the same way, the second denominator is given
  by~$\galg[\cum] \setminus \galg + [\ID]$. For the third denominator we get~$(\evl)$ directly from
  the Diamond Lemma. Applying the second isomorphism theorem to the first and second summand
  yields~\eqref{eq:decomp} since~$[\ID] \cap \galg[\der] = [\Diff]$
  and~$[\ID] \cap \big( \galg[\cum] \setminus \galg \big) = [\RB]$, noting that~$(\evl)/[\ID]$ is
  just~$(\evl) \subset \galg[\der,\cum \Mid \ID]$ in~\eqref{eq:decomp}.

  We give now a proof of~\eqref{eq:decomp-norm}, which follows closely the more general
  argument\footnote{The proofs in~\mcite{RosenkranzRegensburgerTecBuchberger2012} use only
    ring-theoretic properties of~$\bfk$; no field or zero characteristic is required. They are more
    general in that they allow character sets~$\Phi \supsetneq \{ \evl \}.$} given
  in~\mcite{RosenkranzRegensburgerTecBuchberger2012}, specifically Lemma~23 as well as
  Propositions~25 and 26 therein. Let us start by analyzing the irreducible monomials.  We claim
  that each monomial~$w \in \galg[\der,\cum]_{\ID}$ is either of the
  form~$w = f \der^i \in \galg[\der]_{\Diff}$ $(f \in \galg, i \ge 0)$
  or~$w = f \cum g \in \galg[\cum]_{\RB}$ $(f, g \in \galg)$ or $f \cum \der^{i+1}$
  $(f \in \galg, i \ge 0)$. If~$w$ contains any occurrences of~$\der$, they must be in the tail
  of~$w$ since~$\der f \; (f \in \galg)$ is reducible relative to~$[\Diff] \subset \ID$ and
  also~$\der \cum$ relative to~$[\der \cum - 1] \subset [\ID]$. This means we have~$w = v \der^i$
  with prefix monomial~$v \in \galg[\cum]$ and~$i \ge 0$. But then~$v$ can have at most one
  occurrence of~$\cum$ since~$\cum f \cum \; (f \in \galg)$ is reducible relative
  to~$[\RB] \subset [\ID]$. Hence we have either~$w = g \cum f \der^i$ or~$w = f \der^i$ for
  some~$f, g \in \galg$. In the latter case, we obtain~$w \in \galg[\der]_{\Diff}$ and are done.  In
  the former case, we can must have~$i = 0$ or~$f = 1$ since otherwise~$w$ is reducible relative to
  relative
  to~$[\cum f \der - \underline{f} + \cum \underline{f}' + \evl(\underline{f}) \, \evl] \subset
  [\ID]$.
  Hence we have either the case~$w = g \cum f \in \galg[\cum]_{\RB}$, where~$f=1$ is possible.  Or
  else we have the irreducible monomial~$w = g \cum \der^i$ with~$i>0$.
  
  Next we analyze the irreducible elements of~$(\evl)_{\ID}$; unlike those of~$\galg[\der]_{\Diff}$
  and~$\galg[\cum]_{\Diff}$, these are \emph{not} monomials. Since any element of~$(\evl)_{\ID}$ can
  be written as a $\bfk$-linear combination of~$w \evl \tilde{w} \neq 0$ with
  monomials~$w, \tilde{w}$, it suffices to analyze those. As we have seen above, if~$w$ contains any
  occurrences of~$\der$, they must be at its tail. But since~$\der \evl = 0$, there can in fact be
  no~$\der$ in~$w$. By the above analysis of normal forms for~$w$, the only remaining possibilities
  are~$w = f$ and~$w = f \cum g$ for some~$f, g \in \galg$. But the latter is also excluded
  since~$\cum g \evl = \cum g - \cum g \cum \der$ is reducible relative to~$[\RB] \subset \ID$.
  Hence we conclude~$w = f$.  Regarding the monomial~$\tilde{m}$, we it cannot start with
  any~$g \in \galg$ since~$\evl g = g - \cum \der g$ is reducible relative to~$[\Diff] \subset \ID$.
  Furthermore, $\tilde{w}$ cannot start with~$\cum$ since~$\evl \cum = 0$. By our analysis of
  irreducible monomials, this leaves with the only remaining possibility~$\tilde{w} = \der^i$.
  Altogether this show that~$w \evl \tilde{w} = f \evl \der^i$. We may thus conclude that all three
  $\bfk$-modules on the right-hand side of~\eqref{eq:decomp-norm} are in fact left $\galg$-modules
  with the following generators: While~$\galg[\der]_{\Diff}$ is generated by~$\der^i \; (i \ge 0)$,
  and~$\galg[\cum]_{\RB} \setminus \galg$ by~$\cum f \; (f \in \galg)$, the normal forms
  in~$(\evl)_{\ID}$ are generated by~$\evl \der^i$.

  For establishing~\eqref{eq:decomp-norm}, it is sufficient to show that
  each~$U \in \galg[\der,\cum]_{\ID}$ splits uniquely as~$U = U_{\der} + U_{\varint} + U_{\evl}$,
  containing a part~$U_{\der} \in \galg[\der]_{\Diff}$, a
  part~$U_{\varint} \in \galg[\cum]_{\RB} \setminus \galg$, and finally a
  part~$U_{\evl} \in (\evl)_{\ID}$. Each irreducible monomial~$f \der^i$ of~$U$ is put
  into~$U_{\der}$, and each irreducible monomial~$f \cum g$ into~$U_{\varint}$. For irreducible
  monomials of the form~$f \cum \der^{i+1} = f \der^i - f \evl \der^i$, we put~$f \der^i$
  into~$U_{\der}$ and~$-f \evl \der^i$ into~$U_{\evl}$. Thus we
  have~$U = U_{\der} + U_{\varint} + U_{\evl}$; let us now prove uniqueness. Hence
  assume~$\sum_i a_i \der^i + \sum_i b_i \cum c_i + \sum_i d_i \evl \der^i = 0$, each sum having
  finitely many nonzero coefficients~$a_i, b_i, c_j, d_i \in \galg$. By the definition of~$\evl$,
  this is the same
  as~$\sum_i (a_i + d_i) \der^i + \sum_i b_i \cum c_i - \sum_i d_i \cum \der^{i+1} = 0$. By the
  definition of the free operator ring~$\galg[\der,\cum]$, all monomials are linearly independent
  over~$\bfk$, hence~$a_i + d_i = b_i = d_i = 0$ and then also~$a_i = 0$. This completes the
  uniqueness proof for splitting~$U$. We have now established the $\bfk$-module
  decomposition~\eqref{eq:decomp-norm} and therefore also~\eqref{eq:decomp}. Since $\galg[\der]$
  and~$\galg[\cum]$ are both closed under multiplication, they are subalgebras
  of~$\galg[\der,\cum]$.

  Finally, let us prove the quotient
  statement~$\galg[\der, \vcum \Mid \ID] \cong \galg[\der, \vcum \Mid \DRB]/I$, where for once we
  use the same symbol for the Rota-Baxter operator in~$\ID$ and~$\DRB$. (Recall that the notational
  distinction between~$\cum$ and~$\vcum$ is purely a convenience that allows us to suppress the laws
  to be factored out.) Writing out the definitions, we must thus prove
  \begin{equation*}
    \frac{\galg[\der, \vcum]}{[\ID]} \cong \frac{\galg[\der,\vcum]}{[\DRB]} \,\Big/\, [ \evl f -
    \evl(f) \, \evl \mid f \in \galg ] ,
  \end{equation*}
  which reduces to showing~$[\ID]/[\DRB] = [ \evl f - \evl(f) \, \evl \mid f \in \galg ]$ by the
  third isomorphism theorem. Hence it suffices to
  show~$[\ID] = [\DRB] \dotplus [ \evl f - \evl(f) \, \evl \mid f \in \galg ]$ as~$\bfk$-modules,
  where the directness of the sum is obvious. For the inclusion from left to right, we must show
  that every~$\cum f \der - \underline{f} + \cum \underline{f}' + \evl(\underline{f}) \, \evl$ is
  in~$[\ID'] := [\DRB] + [ \evl f - \evl(f) \, \evl \mid f \in \galg ]$.
  Substituting~$f + \lambda f'$ for~$f$, we may also show that
  every~$r_f := \cum (f + \lambda f') \, \der - f + \cum f' + \evl(f) \, \evl$ is in~$[\ID']$. But
  we have indeed
  \begin{equation*}
     r_f = \cum \Big( f \der + \lambda \, f' \der + f' - \der f \Big) - \Big( \evl f - \evl(f) \,
     \evl \Big) \in [\ID']
  \end{equation*}
  since the first summand is in~$\Diff \subset \DRB$ and the second
  in~$[\evl f - \evl(f) \, \evl \mid f \in \galg]$. For the inclusion from right to left, it
  suffices to show that every~$\evl f - \evl(f) \, \evl$ is in~$\ID$. But we have just proved
  that~$r_f + \evl f - \evl(f) \, \evl \in [\Diff] \subset [\ID]$. Since we have also~$r_f \in \ID$,
  the proof is completed.
\end{proof}

As the name suggests, there is another important aspect to $E$-operators that we should consider
here---they \emph{operate} on suitable domains. These domains are a special class of modules that we
shall now introduce. Recall first that an \emph{$\uo$-operated
  module}~$(M; p_\omega \Mid \omega \in \uo)$ over a commutative ring~$A$ is an $A$-module~$M$ with
$A$-linear operators~$p_\omega\colon M \to M$. As in the case of~$\uo$-operated algebras, no
restrictions are imposed on the operators~$p_\omega$. An operated
morphism~$\phi\colon (M; p_\omega \Mid \omega \in \uo) \to (M'; p_\omega' \Mid \omega \in \uo)$ is
an $A$-linear homomorphism~$\phi\colon M \to M'$ such
that~$\phi \circ p_\omega = p_\omega' \circ \phi$ for all~$\omega \in \uo$; the resulting category
of~$\uo$-operated modules over~$A$ is denoted by~$\Mod_A(\uo)$.

Now assume that~$A \in \CAlg(\uo)$ is an operated algebra. Then the
free operators~$T \in \opra$ \emph{act naturally} on the
$\uo$-operated $A$-module~$M$. Since~$T$ is a $\bfk$-linear
combination of noncommutative monomials~$t \in M(\uo \uplus A)$, it
suffices to define~$t \cdot m$ for~$m \in M$. By the universal
property of~$M(\uo \uplus A)$, we obtain a unique monoid action by
setting~$\omega \cdot m := p_\omega(m)$ for~$\omega \in \uo$ and~$a
\cdot m := am$ for~$a \in A$. Thus~$M$ becomes an~$\opra$-module, and
we can now introduce the module-theoretic analog of $E$-algebras.

\begin{defn}
  Fix an operated algebra~$A \in \CAlg(\uo)$ and a submodule~$E
  \subseteq \bfk\frakC_\Omega(Y)$ of standard laws. Then
  an~\emph{$E$-related module} over~$A$ is an operated module~$M \in
  \Mod_A(\uo)$ with~$L \cdot m = 0$ for all relators~$L \in [E]$
  and~$m \in M$.
\end{defn}

Again we will briefly speak of $E$-modules (since the context will make it clear that~$E$ is a set
of laws). They form a full subcategory of~$\Mod_A(\uo)$ denoted by $\Mod_A(\uo|E)$. The role of the
$E$-operator ring becomes clear now: Operators correspond to the natural action defined above if~$A$
is an $E$-algebra. This can be made precise by the following statement.

\begin{prop}
  Let~$A$ be an $E$-algebra for a submodule~$E \subseteq
  \bfk\frakC_\Omega(Y)$. Then we have the isomorphism of
  categories~$\Mod_A(\uo|E) \cong \Mod_{\eopra}$.
  \mlabel{prop:op-module}
\end{prop}

\begin{proof}
  As noted above, an $E$-module~$M \in \Mod_A(\uo|E) \subseteq
  \Mod_A(\uo)$ can also be viewed as an~$\opra$-module under the
  natural action, and as such it satisfies~$[E] \cdot M = 0$. But then
  the action of~$\eopra$ with~$(T + [E]) \cdot m := T \cdot m$ is
  well-defined and gives~$M$ the structure of
  an~$\eopra$-module. Conversely, every such module restricts to an
  operated module~$M \in \Mod_A(\uo)$ with~$[E] \cdot M = 0$.

  Of course, every morphism of~$\Mod_{\eopra}$ is also a morphism
  of~$\Mod_A(\uo|E)$. For the other direction, let~$\phi$ be a
  morphism of $E$-modules. For showing~$\phi\big( (T+[E]) \cdot m
  \big) = \big( T + [E] \big) \cdot \phi(m)$ for~$T \in \opra$ and $m
  \in M$, it suffices to show~$\phi(T \cdot m) = T \cdot
  \phi(m)$. Since~$\phi$ is $\bfk$-linear, we may assume a monomial~$T
  \in M(\uo \uplus A)$ and use induction on the degree of~$T$. The
  base case~$T = 1$ is trivial, hence assume the claim for monomials
  of degree~$n$ and let~$T$ have degree~$n+1$. Then there exists~$T'
  \in M(\uo \uplus A)$ of degree~$n$ such that either~$T = a T'$
  for~$a \in A$ or~$T = \omega T'$ with~$\omega \in \uo$. In the
  former case the claim follows because~$\phi$ is~$A$-linear, in the
  latter case because it is a morphism of~$\uo$-operated modules. This
  completes the proof that~$\phi$ is also a morphism
  of~$\Mod_{\eopra}$.
\end{proof}

\begin{fact}
  \label{ft:standard-Emodules}
  Some standard constructions for creating new modules also work in the operated setting. Let us
  mention a few that are also relevant for the examples to be given afterwards.
  \begin{enumerate}
  \item\label{it:freemod} If~$A$ is an $E$-algebra and~$S$ an arbitrary set, the \emph{free
      module}~$A^S$ is an $E$-module. The action of~$p_\omega \; (\omega \in \Omega)$ on a module
    element~$f \in A^S$ is defined by~$(p_\omega f)(s) := P_\omega(fs)$ for~$s \in S$. It is easy to
    see that for any free operator~$L \in A[\Omega]$ and~$f \in A^S$ one
    has~$(L \cdot f)(s) = L \cdot f(s)$ for all~$s \in S$, where the left action takes place
    in~$A^S$ and the right action in~$A$. Hence one obtains~$L \cdot f = 0$ for all
    relators~$L \in [E] \subset A[\Omega]$ and all~$f \in A^S$, which confirms that~$A^S$ is an
    $E$-module. Note that~$A^S$ is free as an $A$-module but generally not as an $\eopra$-module
    (see Example~\ref{ex:vmod}\ref{it:diffmod} below).
  \item\label{it:dirprod} If~$M_1, \dots, M_k$ are $E$-modules over~$A$, their direct
    product~$M_1 \times \cdots \times M_k$ is an $E$-module with
    operators~$p_\omega \; (\omega \in \Omega)$ acting component-wise. If~$M_1 = \cdots = M_k = M$,
    this gives the free module~$M^S$ over the finite set~$S = \{1, \dots, k\}$.
  \item\label{it:dualmod} Whenever~$M$ is an $E$-module over~$A$, the \emph{dual module}~$M^*$ is
    naturally an $E^*$-module with operators~$p_\omega^* \; (\omega \in \Omega)$;
    here~$p_\omega^*\colon M^* \to M^*$ is defined as the dual map of the $A$-linear
    map~$p_\omega\colon M \to M$. If~$L \in A[\Omega]$ is any free operator and~$f \in M^*$ one
    checks immediately that~$(L^* \cdot f)(m) = f(L \cdot m)$ for all~$m \in M$. In other words, the
    action on~$M^*$ is the dual of the action on~$M$. In particular, one sees that~$L^* \cdot f = 0$
    for all relators~$L \in [E] \subset A[\Omega]$ and all~$f \in M^*$, confirming that~$M^*$ is
    $E^*$-related (meaning it satisfies the transpose of all relators induced by~$E$). Since
    any~$E$-algebra~$A$ is also an $E$-module over itself, $A^*$ is also an $E^*$-module.
  \item\label{it:submod} If~$M$ is an $E$-module with a submodule~$M' \subseteq M$ that is closed
    under all all operators~$p_\omega \; (\omega \in \Omega)$, their restrictions to the
    submodule~$M'$ make the latter into an $E$-module or, more precisely, an \emph{$E$-related
      submodule} of~$M$.
  \end{enumerate}
\end{fact}

\begin{exam}
  \label{ex:vmod}
  Let us now exemplify the concept of $E$-module for the \emph{four
    standard varieties} given in Example~\ref{ex:lin-var},
  corresponding to the four operator rings of
  Proposition~\ref{prop:lin-op-rings}. We make again use of the
  convention stated in Remark~\ref{rem:suppress-var}.
  \begin{enumerate}
  \item\label{it:diffmod} The~$\Diff$-modules are commonly known as \emph{differential
      modules}~\mcite[Def.~1.2.4(iii)]{Singer2009}, usually taken with weight~$\lambda=0$ over a
    differential field~$(\galg, \der)$. Their equivalent formulation as~$\galg[\der]$-modules is
    often used as an alternative definition~\mcite[Def.~2.5]{PutSinger2003}. Differential modules
    are crucial for differential Galois theory as they provide an abstract way of formulating linear
    differential equations. In the important special case when the underlying differential ring
    is~$\galg = \bfk[x]$, the operator ring is the Weyl algebra~$A_1(\bfk)$, and the corresponding
    $A_1(\bfk)$-modules are known as $\mathcal{D}$-modules~\mcite{Coutinho1995}
    since~$\mathcal{D} := A_1(\bfk) = \bfk[x][\der]$ is the underlying differential operator ring.
    For example, $\bfk[x]^n$ is a differential module by
    Fact~\ref{ft:standard-Emodules}\ref{it:dirprod}. If~$\bfk$ is a field, any
    $\bfk$-basis~$e_1, \dots, e_n$ of~$\bfk^n$ is of course a $\bfk[x]$-basis for~$\bfk[x]^n$ but
    since~$\der e_1, \dots, \der e_n = 0$ it is not an $A_1(\bfk)$-basis. In other words,
    $\bfk[x]^n$ is free as a $\bfk[x]$-module but not as a $\bfk[x][\der]$-module.

    Another important class of examples with~$\bfk = \RR$ is concerned with \emph{vector fields} on
    a manifold~$M$. In detail, each vector field~$V \in \mathfrak{X}(M)$ induces a covariant
    derivative~$\nabla_V\colon \mathfrak{X}(M) \to \mathfrak{X}(M)$ with characteristic
    property~$\nabla_V(f W) = f' W + f \, \nabla_V(W)$ for~$f \in C^\infty(M)$
    and~$W \in \mathfrak{X}(M)$. The vector fields~$\mathfrak{X}(M)$ thus form a differential module
    over the differential algebra~$C^\infty(M)$.
  \item\label{it:rotmod} The category of~$\RB$-modules has been
    introduced in~\mcite[Def.~2.1(a)]{GaoGuoLin2015} under the name
    of~\emph{Rota-Baxter module} for a given Rota-Baxter
    algebra~$(\galg, \cum)$. Their equivalent description in terms
    of~$\galg[\cum]$-modules is elaborated
    in~\mcite[\S2.2]{GaoGuoLin2015}.
  \item\label{it:dist} Similarly, we introduce now the category of~$\DRB$-modules, which we may also
    call \emph{differential Rota-Baxter modules}. In $\mathcal{D}$-module theory\footnote{The two
      occurrences of~$\mathcal{D}$ in ``$\mathcal{D}$-modules'' and in the distribution
      space~$\mathcal{D}'(\RR)$ are unrelated. In fact, $\mathcal{D}'(\RR)$ is the dual of the
      differentiable class~$\mathcal{D}(\RR) := C^\infty_0(\RR)$ of smooth functions with compact
      support.}, it is often pointed out that various spaces of (real or complex valued)
    distributions are differential modules (for weight~$\lambda = 0$ and ground field~$\bfk = \RR$
    or~$\bfk = \CC$) and hence~$\mathcal{D}$-modules since distributions can be multiplied by smooth
    functions so they are in particular modules over~$\galg := \bfk[x]$. It is seldom appreciated
    that some of these distribution spaces are in fact differential Rota-Baxter modules over~$\galg$
    and hence $\galg[\der,\vcum]$-modules. For example,
    let~$\mathcal{D}'(\RR)_+ \subset \mathcal{D}'(\RR)$ be the space of all distributions~$T$ with
    \emph{left}-bounded support, meaning~$\supp(T) \subseteq [a, \infty[$ for some~$a \in \RR$.
    Analogously, we write~$\mathcal{D}(\RR)_-$ for the space of test functions with
    \emph{right}-bounded support; it is clear that this is a (non-unital!) differential Rota-Baxter
    algebra with standard derivation~$\der$ and Rota-Baxter operator~$\vcum := -\cum_{\!x}^\infty$.
    In fact, $\mathcal{D}(\RR)_-$ is a degenerate (nonunital) integro-differential algebra since the
    induced evaluation~$\evl := 1_{\mathcal{D}(\RR)_+} - \vcum \circ \der = 0$ is trivially
    multiplicative. In other words, $\der$ is bijective with~$\vcum$ as its inverse, and the strong
    Rota-Baxter axiom~$f \vcum g = \vcum f' \vcum g + \vcum fg$ immediately follows
    from~$f' \vcum g = (f \vcum g)' - fg$. If~$H \in \mathcal{D}'(\RR)_+$ is the Heaviside function,
    the operator~$\vcum\colon \mathcal{D}'(\RR)_+ \to \mathcal{D}'(\RR)_+$ defined by the
    convolution~$\vcum T := H \star T$ is known to be a two-sided
    inverse~\cite[\S13.1]{DuistermaatKolk2010} of the distributional derivative~$\der$. One checks
    that~$\vcum\colon \mathcal{D}'(\RR)_+ \to \mathcal{D}'(\RR)_+$ is the transpose of
    $\vcum\colon \mathcal{D}(\RR)_- \to \mathcal{D}(\RR)_-$, just as the distributional
    derivative~$\der\colon \mathcal{D}'(\RR)_+ \to \mathcal{D}'(\RR)_+$ is (by definition) the
    transpose of the standard derivation~$\der\colon \mathcal{D}(\RR)_- \to \mathcal{D}(\RR)_-$.
    Thus we obtain a Rota-Baxter module~$(\mathcal{D}'(\RR)_+, \der, \vcum)$, which is actually a
    degenerate integro-differential module over the nonunital Rota-Baxter
    algebra~$\mathcal{D}(\RR)_-$. Of course, one may apply a similar construction to endow the
    space~$\mathcal{D}'(\RR)_-$ of \emph{right}-bounded distributions with the structure of a
    differential Rota-Baxter module over the nonunital differential Rota-Baxter algebra of
    \emph{left}-bounded test functions.
  \item\label{it:intdiffmod} Finally, let us consider the category of~$\ID$-modules, which we also
    call \emph{integro-differential modules}. Again we shall give an important example from
    distribution theory. Endowing~$\mathcal{E}(\RR) := C^\infty(\RR)$ with the usual
    derivation~$\der = d/dx$ and the Rota-Baxter operator~$\vcum f := \cum_{\!0}^x f(\xi) \, d\xi$
    yields a ``dually integro-differential'' module $\mathcal{E}^*(\RR)$ by
    Fact~\ref{ft:standard-Emodules}\ref{it:dualmod}, as the dual of the integro-differential
    algebra~$(\mathcal{E}(\RR), \der, \vcum)$. Just as the one-sided distribution spaces of
    Item~\ref{it:dist}, this is in fact a differential module as well as a Rota-Baxter module since
    the relators~$\der f \to f \der + f'$
    and~$\vcum f \vcum \to f\backquote \vcum - \vcum f\backquote$ are skew-symmetric under
    transposition. Hence we should take the \emph{negated} transposes
    of~$\der, \cum\colon \mathcal{E}(\RR) \to \mathcal{E}(\RR)$; this is of course standard practice
    in defining the distributional derivative~\cite[\S4]{DuistermaatKolk2010}. As the two signs
    cancel, we obtain the \emph{transposed section law} $\vcum \circ \der = 1_{\mathcal{E}'(\RR)}$.

    One checks that both~$\der$ and~$\vcum$ restrict to the topological
    dual~$\mathcal{E}'(\RR) \subset \mathcal{E}^*(\RR)$ consisting of all continuous
    functionals~$\mathcal{E}(\RR) \to \bfk$, relative to the well-known locally convex topology
    of~$\mathcal{E}(\RR)$; see for example~\cite[(7.8)]{Szmydt1977}. Therefore~$\mathcal{E}'(\RR)$
    is a differential Rota-Baxter submodule of~$\mathcal{E}^*(\RR)$ by
    Fact~\ref{ft:standard-Emodules}\ref{it:submod}, except that the section law is transposed. In
    analysis, $\mathcal{E}'(\RR)$ is known as the space of \emph{compactly supported
      distributions}. It may seem surprising that~$\der$ is injective and~$\vcum$ surjective
    on~$\mathcal{E}'(\RR)$. In fact, one checks that~$\ker{\vcum} = \RR \delta_0$
    and~$\im(\der) = \{ T \in \mathcal{E}'(\RR) \mid T(1) = 0 \}$. It is
    known~\cite[(10.4)]{Szmydt1977} that in~$\mathcal{D}'(\RR) \supset \mathcal{E}'(\RR)$, the
    kernel of~$\der$ is given by the constant distributions; but since their support is~$\RR$, they
    are not in~$\mathcal{E}'(\RR)$. Conversely, the image of~$\der$ on~$\mathcal{D}'(\RR)$ is
    full~\cite[Cor.~\S11.2]{Szmydt1977} since the constant function~$1$ is not compactly supported
    so the condition~$T(1) = 0$ is void. As we have seen, $\der$ is not surjective
    on~$\mathcal{E}'(\RR)$, which is of course well-known~\cite[Ex.~11.2]{Szmydt1977}.

    Is~$(\mathcal{E}'(\RR), \der, \vcum)$ an integro-differential module (with transposed section
    law)? One must check if the (transposed) induced
    evaluation~$\evl := 1_{\mathcal{E}'(\RR)} - \der \circ \vcum\colon \mathcal{E}'(\RR) \to
    \mathcal{E}'(\RR)$
    is multiplicative in the sense that~$\evl(f T) = \evl(f) \, \evl(T)$ for
    all~$T \in \mathcal{E}'(\RR)$ and~$f \in \mathcal{E}(\RR)$.
    Since~$\evl\colon \mathcal{E}'(\RR) \to \mathcal{E}'(\RR)$ is the transpose
    of~$\evl\colon \mathcal{E}(\RR) \to \mathcal{E}(\RR)$, one has~$\evl(T) = T(1) \, \delta_0$ for
    any~$T \in \mathcal{E}'(\RR)$. But then one sees that
    $$\evl(e^x \, \delta_1) = \delta_1(e^x) \cdot \delta_0 = e \cdot \delta_0 \neq \evl(e^x) \,
    \evl(\delta_1) = 1 \cdot \delta_0,$$
    which shows that~$\mathcal{E}'(\RR)$ is in fact \emph{not} an integro-differential module. The
    problem is that the corresponding relator~$\evl f \to \evl(f) \, \evl$ gets transposed
    to~$f \evl \to \evl(f) \, \evl$, which yields the true identity~$f \, \evl(T) = f(0) \, \evl(T)$
    or~$f \, T(1) \, \delta_0 = f(0) \, T(1) \, \delta_0$. If~$T(1) = 0$, this is trivially valid;
    otherwise division by~$T(1)$ yields the familiar \emph{sifting property} of the Dirac
    distribution~\cite[p.~38]{DuistermaatKolk2010}. Of course we may replace~$\delta_0$
    by~$\delta_c$ for any~$c \in \RR$ if we use the Rota-Baxter operator~$\cum_{\!a}^x$
    on~$\mathcal{E}(\RR)$ instead of~$\cum_{\!0}^x$.
  \end{enumerate}
\end{exam}

For seeing an honest integro-differential module, we refer to~\cite{RosenkranzSerwa2017}, where the
algebraic \emph{distribution module}~$(\mathcal{D}\galg, \sder, \scum)$ over a given ordinary
shifted integro-differential algebra~$\galg$, such as the classical example~$\galg = C^\infty(\RR)$,
is constructed and investigated. This provides a purely algebraic structure (involving no topology,
in particular taking~$\galg$ only as an integro-differential algebra with shift maps such
as~$f(x) \mapsto f(x-c)$ for~$c \in \RR$ in the classical example), providing just piecewise
functions and Dirac distributions on top of~$\galg$. In the classical example, this gives rise to
the Heaviside function~$H_a = H(x-a)$ and their derivatives~$\delta_a$. Compared to the analytic
distribution spaces of Example~\ref{ex:vmod} \ref{it:dist}, \ref{it:intdiffmod}, this is a very
small module. However, it contains exactly what is needed for specifying and computing the
\emph{Green's operator}~$G \in \galg[\der, \cum]$ of a LODE boundary
problem~\cite[\S\S2,3]{Stakgold1979}. Acting as~$G\colon \galg \to \galg$, it can be assigned a
\emph{Green's function} $g(x,\xi)$. This is a (bivariate) function involving Heavisides and---for
ill-posed problems--- also Diracs, characterized by a distributional differential equation. For a
comprehensive presentation, we refer the reader to~\cite{RosenkranzSerwa2017}, specifically
Theorems~26 and~29 therein. The actual computation of the Green's operator~$G$ on the basis of a
given fundamental system is detailed in~\cite{RosenkranzRegensburger2008a}, the extraction of the
Green's function~$g(x,\xi)$ from~$G$ in~\cite{RosenkranzSerwa2015}.

% =============================================================================
\section{Differential Rota-Baxter Operators}
\label{sec:drbo}
% =============================================================================

As pointed out earlier, the operator rings in Proposition~\ref{prop:lin-op-rings}\ref{it:diffop},
\ref{it:intop}, \ref{it:intdiffop} are known and defined elsewhere, but the ring
in~\ref{it:diffrbop} is introduced here for the first time. In the rest of this paper, we will
therefore concentrate on the ring~$\galg[\der,\vcum]$. As a first step, let us analyze its canonical
forms, in a way similar to~\mcite[Prop.~25]{RosenkranzRegensburgerTecBuchberger2012} and the above
$\bfk$-module decomposition for~$\galg[\der,\cum]$. In the following, recall that
$\galg[\vcum] \!\setminus\!  \galg$ denotes a linear complement rather than the set-theoretic one.

\begin{lemma}
  \label{lem:drbo-decomp}
  Let~$(\galg, \der, \vcum) \in \DRB$. Then we have~$\galg[\der,\vcum]
  = \galg[\der] \dotplus \galg[\vcum] \!\setminus\! \galg \dotplus
  [\evl]$, where~$[\evl] := \bfk \{ f \vcum g \, \der^k \mid f, g \in
  \galg; \: k > 0 \}$ is a rung that we call the \emph{evaluation
    rung}.
\end{lemma}

\begin{proof}
  The proof of the direct sum is completely analogous to that of the corresponding statement in
  Proposition~\ref{prop:lin-op-rings}, with~\eqref{eq:decomp-norm} being replaced by
  \begin{equation}
    \label{eq:decomp-drb}
    \galg[\der,\vcum]_{\DRB} = \galg[\der]_{\Diff} \dotplus \galg[\vcum]_{\RB} \setminus \galg
    \dotplus [\evl]_{\DRB} .
  \end{equation}
  The analysis of irreducible monomials~$w \in \galg[\der,\vcum]_{\DRB}$ is also the same, except
  that the remaining case~$w = f \vcum g \der^k$ with~$f, g \in \galg$ and~$k \ge 0$ cannot be
  reduced any further. We have of course~$w \in [\evl]$ if~$k > 0$ and~$w \in \galg[\vcum]_{\RB}$
  otherwise. The direct sum~\eqref{eq:decomp-drb} now follows immediately since the evaluation
  rung~$[\evl]$, unlike the evaluation ideal~$(\evl)$, is generated by irreducible monomials.

  It remains to prove that~$[\evl]$ is multiplicatively closed,
  meaning~$(f \vcum g \, \der^k)(\tilde{f} \vcum \tilde{g} \,
  \der^{\tilde{k}}) \in [\evl]$. It suffices to ensure~$w_k := \vcum g
  \, \der^k \tilde{f} \vcum \tilde{g} \, \der \in [\evl]$, and for
  that we use induction over~$k>0$. For~$k=1$ we have~$\der \tilde{f}
  = \tilde{f}'+ \tilde{f} \der + \lambda \, \tilde{f}' \der$ and~$w_k
  = \vcum \tilde{f}' g \vcum \tilde{g} \, \der + \vcum \tilde{f} g
  \tilde{g} \, \der + \lambda \, \vcum \tilde{f}' g \tilde{g} \der \in
  [\evl]$ after applying the Rota-Baxter rule in the first
  summand. For the induction step we consider~$w_{k+1}$, assuming the
  claim holds for~$k$. We obtain
  \begin{equation*}
    w_{k+1} = \vcum g \der^k \tilde{f}' \vcum \tilde{g} \, \der +
    \vcum g \der^k \tilde{f} \tilde{g} \der + \lambda \, \vcum g
    \der^k \tilde{f}' \tilde{g} \der,
  \end{equation*}
  where the first summand is contained in~$[\evl]$ by the induction
  hypothesis and the second expands into a linear combination of terms
  having the shape~$\tilde{w}_k \der^l \: (1 \le l \le k+1)$, which
  are clearly contained in~$[\evl]$ as well.
\end{proof}

Note that both~$\galg[\vcum]_+ := \galg[\vcum] \!\setminus\! \galg$
and~$\galg[\der]_+ := \galg[\der] \!\setminus\! \galg$ are rungs,
which feature in the alternative $\bfk$-module
decomposition~$$\galg[\der,\vcum] = \galg \dotplus \galg[\der]_+
\dotplus \galg[\vcum]_+ \dotplus [\evl].$$ Moreover, one checks
immediately that~$[\evl]$ is actually an~$(\galg[\vcum]_+,
\galg[\der]_+)$-bimodule. According to the subsequent lemma, the
evaluation rung is also closely related to the evaluation (hence its
name). Note that we continue to call the projector~$\evl := 1_\galg -
\vcum \der$ the \emph{evaluation} of the differential Rota-Baxter
algebra~$(\galg, \der, \vcum)$ although it is not multiplicative
(unless~$\galg$ is in fact an integro-differential algebra). However,
it is still a projector onto~$\ker{\der}$ along~$\im{\cum}$. By abuse
of language, the corresponding~$\evl \in \galg[\der,\vcum]$ will also
be referred to as evaluation.

\begin{lemma}
  The evaluation rung~$[\evl]$ is a bimodule over~$\bfk[\evl]$, with
  $\evl$ as right annihilator.
\end{lemma}

\begin{proof}
  Since~$\evl^2 = \evl$, the ring~$\bfk[\evl]$ is the~$\bfk$-span
  of~$1$ and~$\evl$. Therefore it suffices to verify the
  inclusion~$\evl \, [\evl] \subseteq [\evl]$ and~$[\evl] \, \evl =
  0$. The latter is immediate from the definition of~$[\evl]$, the
  former follows from~$\evl \, f \vcum g \, \der^k = [ \evl(f), \vcum]
  \, g \der^k \in [\evl]$ via the Rota-Baxter axiom; here the bracket
  denotes the commutator in~$\galg[\der,\vcum]$.
\end{proof}

Before we study further properties of \emph{differential Rota-Baxter
  algebras} and their operator rings, let us give two simple examples
(the weight is zero for both).

\begin{exam}
  \label{ex:pol}
  Let~$\bfk$ have characteristic zero. The most basic example of a
  differential Rota-Baxter algebra is clearly the polynomial
  ring~$\bfk[x]$, with standard derivation~$\der = d/dx$ and
  Rota-Baxter operator~$\cum = \cum_0^x$ or more generally~$\cum_a^x$
  for any initialization point~$a \in \bfk$. Here we think
  of~$\cum_a^x\colon \bfk[x] \to \bfk[x]$ in purely algebraic terms,
  as the~$\bfk$-linear map defined by~$x^k \mapsto
  (x^{k+1}-a^{k+1})/(k+1)$. This example will play a great role in
  Section~\ref{sec:intdiff-weyl} although it is not a genuine example
  (in the sense that it is also an integro-differential algebra).
\end{exam}

\begin{exam}
  \label{ex:pcw-cont}
  For seeing a natural example of a differential Rota-Baxter algebra
  that is not an integro-differential algebra, we call on analysis. Of
  course, the primordial example of an integro-differential algebra
  consists of the (real or complex valued) \emph{smooth
    functions}~$C^\infty(\RR)$ or~$C^\infty[a,b]$;
  see~\mcite[Ex.~5]{RosenkranzRegensburger2008a}. Here~$\der$
  and~$\cum = \cum_\xi^x \; (\xi \in \RR$ or~$\xi \in [a,b]$) are
  defined analytically.

  A slight variation of this example leads to a differential
  Rota-Baxter algebra, namely the (real or complex valued)
  \emph{piecewise smooth functions}~$PC^\infty(\RR)$
  or~$PC^\infty[a,b]$. For example, we take all functions that are
  smooth on the whole domain minus finitely many points. The
  operations are defined as before except that~$\der f$ and~$\cum f$
  is undefined at the points where~$f$ is so. (The ring operations $+,
  -, *$ have to be defined carefully since singularities may cancel;
  the result is always to be taken with all removable singularities
  actually removed. This process is also well-known in complex
  analysis where meromorphic functions can be defined in a similar
  way.)

  The piecewise smooth functions are clearly a \emph{differential Rota-Baxter algebra}. However,
  they are not an integro-differential algebra for if they were, the evaluation~$1-\der\cum$ would
  be multiplicative---which it cannot be for functions undefined on the initialization
  point~$\xi$. For a more explicit example, let us take~$PC^\infty[0,1]$ with initialization
  point~$\xi = 0$. The Heaviside function~$h(x):=H(x-1/2) \in PC^\infty[0,1]$ is the characteristic
  function of the subinterval~$[1/2,1]$, and we have~$\cum h = \smash{\cum_{\!1/2}^x} \, dx = x-1/2$
  but~$h \cdot \cum 1 = H(x-1/2) \, x$. This means we have~$\cum (h \cdot 1) \neq h \cdot \cum 1$
  although~$h \in \ker{\der}$, and~\mcite[Rem.~2.6(c)]{GuoRegensburgerRosenkranz2012} shows
  that~$(PC^\infty[0,1], \der, \cum)$ is not an integro-differential algebra.
\end{exam}

In Proposition~\ref{prop:lin-op-rings} the \emph{relation between the
  operator rings}~$\galg[\der,\vcum]$ and~$\galg[\der,\cum]$ is
illuminated in one direction only: It shows the differential
Rota-Baxter operators~$\galg[\der,\vcum]$ to have a finer structure
from which one obtains the integro-differential operator
ring~$\galg[\der,\cum]$ as a quotient. However, we shall see below
(Proposition~\ref{thm:gen-isom}) that the finer
ring~$\galg[\der,\vcum]$ can also be embedded into a suitably
``generic'' integro-differential operator ring. Applying this to the
special case of polynomial coefficients will enable us to give an
operator-theoretic interpretation to the integro-differential Weyl
algebra (Section~\ref{sec:intdiff-weyl}).

As a preparation to this construction, let us first determine the
\emph{free integro-differential algebra}~$(\tilde\galg, \der, \vcum)$
over a given differential Rota-Baxter algebra~$(\galg, \der,
\cum)$. In other words, we want to ``extend''~$\galg$ just enough to
build an integro-differential structure. Categorically speaking, the
association~$\galg \mapsto \tilde\galg$ is the left adjoint of the
forgetful functor~$\ID \to \DRB$. However, note that~$\cum$ is not an
extension of~$\vcum$.

\begin{prop}
  \label{prop:free-intdiffalg}
  Given~$(\galg, \der, \vcum) \in \DRB$, construct~$\tilde{\galg} =
  \galg \otimes_K \galg$ over~$K := \ker{\der}$, extending the
  derivation to~$\der\colon \tilde{\galg} \to \tilde{\galg}$, $f
  \otimes f_\epsilon \mapsto (\der f) \otimes f_\epsilon$ and
  defining~$\cum\colon \tilde{\galg} \to \tilde{\galg}$ via $\cum \, f
  \otimes f_\epsilon := (\vcum f) \otimes f_\epsilon - 1 \otimes
  (f_\epsilon \vcum f)$. Then one obtains~$(\tilde{\galg}, \der, \cum)
  \in \ID$ with evaluation $\evl(f \otimes f_\epsilon) = 1 \otimes f
  f_\epsilon$, and an embedding~$\iota\colon \galg \to \tilde\galg, f
  \mapsto f \otimes 1$ of differential algebras.
\end{prop}

\begin{proof}
  Let us first reassure ourselves that~$\der\colon \tilde{\galg} \to \tilde{\galg}$ is
  well-defined. It suffices to prove that~$\sum_i f(i) \otimes f_\epsilon(i) = 0$
  implies~$\sum_i f(i)' \otimes f_\epsilon(i) = 0$ for finite
  families~$f(i), f_\epsilon(i) \in \galg$. Hence assume~$\sum_i f(i) \otimes f_\epsilon(i) = 0$.
  By~\mcite[Lem.~6.4]{Eisenbud1995}, there are~$c(i,j) \in K$ and~$g(i) \in \galg$ satisfying the
  relations $\sum_j c(i,j) \, g(j) = f(i)$ for all~$i$ and~$\sum_i c(i,j) \, f_\epsilon(i) = 0$ for
  all~$j$, which
  yields~$\sum_i f(i)' \otimes f_\epsilon(i) = \sum_{ij} c(i,j) \, g(j)' \otimes f_\epsilon(i) =
  \sum_j g(j)' \otimes \big( \sum_i c(i,j) \, f_\epsilon(i) \big) = 0$
  since~$c(i,j)' = 0$.  Next we note that~$\iota\colon \galg \to \tilde\galg$ is injective since its
  image is~$\galg \otimes_K K \cong \galg$; it is a morphism of~$\Diff$
  because~$\der (f \otimes 1) = f' \otimes 1$.

  The same argument can be used to demonstrate that~$\cum\colon \tilde{\galg} \to \tilde{\galg}$ is
  well-defined. Hence assume~$\sum_i f(i) \otimes f_\epsilon(i) = 0$ as before; we
  show~$\sum_i \big( \vcum f(i) \big) \otimes f_\epsilon(i) = \sum_i 1 \otimes \big( f_\epsilon(i)
  \vcum f(i) \big)$.
  Since~$\tilde{c}(i,j) := \vcum c(i,j) \, g(j) - c(i,j) \vcum g(j) \in K$, we
  get~$\tilde{c}(i,j) \otimes f_\epsilon(i) = 1 \otimes \tilde{c}(i,j) \, f_\epsilon(i)$ and
  therefore
  \begin{equation*}
    \sum_i \big( \vcum f(i) \big) \otimes f_\epsilon(i) = \sum_{i,j} \tilde{c}(i,j) \otimes
    f_\epsilon(i) = \sum_{i,j} 1 \otimes \tilde{c}(i,j) \, f_\epsilon(i) = \sum_i 1 \otimes \big(
    f_\epsilon(i) \vcum f(i) \big),
  \end{equation*}
  where the first equality
  uses~$\sum_{i,j} \big( c(i,j) \, \vcum g(j) \big) \otimes f_\epsilon(i) = \sum_j \big( \vcum g(j)
  \big) \otimes \big( \sum_i c(i,j) \, f_\epsilon(i) \big) = 0$
  and the
  last~$\sum_{i,j} 1 \otimes \big( c(i,j) \, \vcum g(j) \big) f_\epsilon(i) = \sum_j 1 \otimes \big(
  \sum_i c(i,j) \, f_\epsilon(i) \big) \vcum g_j = 0$.

  Using now the fact that~$\vcum\colon \galg \to \galg$ is a Rota-Baxter operator, a short
  calculation reveals that~$\cum\colon \tilde\galg \to \tilde\galg$ is as well. Moreover, it is
  immediate that~$\der \cum = 1_{\tilde\galg}$, so~$(\tilde\galg, \der, \cum)$ is at least a
  differential Rota-Baxter algebra. Its evaluation is given by
  \begin{align*}
    \evl \, (f \otimes f_\epsilon) &= f \otimes f_\epsilon - \cum \, f'
    \otimes f_\epsilon = f \otimes f_\epsilon - (\vcum f') \otimes
    f_\epsilon + 1 \otimes (f_\epsilon \vcum f')\\
    &= \evl_\galg(f) \otimes f_\epsilon + 1 \otimes (f_\epsilon \vcum
    f') = 1 \otimes f_\epsilon \big(\evl_\galg(f) + \vcum f'\big)
    = 1 \otimes f_\epsilon f,
  \end{align*}
  where in the third step we have used the definition of the
  evaluation on~$\galg$ and in the fourth the fact that all tensors
  are over~$K = \ker{\der} = \im{\evl}$. From this we see that the
  evaluation on~$\tilde{\galg}$ is multiplicative, which
  implies~$(\tilde\galg, \der, \cum) \in \ID$ by
  \mcite[Thm.~2.5(b)]{GuoRegensburgerRosenkranz2012}.
\end{proof}

\begin{theorem}
  \label{thm:free-intdiff}
  The integro-differential algebra~$\tilde\galg$ defined in
  Proposition~\ref{prop:free-intdiffalg} is free over~$\galg$. In
  other words, any~$\DRB$-morphism $\phi\colon \galg \to \ogalg$ to an
  integro-differential algebra~$\ogalg$ factors as~$\phi = \tilde\phi
  \circ \iota$ for a unique $\ID$-morphism~$\tilde\phi\colon
  \tilde\galg \to \ogalg$.
\end{theorem}

\begin{proof}
  Let us first prove uniqueness of~$\tilde\phi$. Assuming~$\phi = \tilde\phi \circ \iota$, we
  have~$\tilde\phi(f \otimes 1) = \phi(f)$. Moreover,
  $1 \otimes f_\epsilon = \evl(f_\epsilon \otimes 1)$
  implies~$\tilde\phi(1 \otimes f_\epsilon) = \evl_\ogalg \smash{\big( \tilde\phi(f_\epsilon \otimes
    1) \big)} = \evl_\ogalg \smash{\big( \phi(f_\epsilon) \big)}$
  since~$\tilde\phi$ is an~$\ID$-morphism and thus commutes with the evaluation. As~$\tilde\phi$ is
  a morphism of $\bfk$-algebras we
  obtain~$\tilde\phi(f \otimes f_\epsilon) = \phi(f) \, \evl_\ogalg \smash{\big( \phi(f_\epsilon)
    \big)}$, which determines~$\tilde\phi\colon \tilde\galg \to \ogalg$ uniquely.

  For proving existence, it suffices to show that
  defining~$\tilde\phi(f \otimes f_\epsilon) := \phi(f) \, \evl_\ogalg \smash{\big( \phi(f_\epsilon)
    \big)}$
  yields an~$\ID$-morphism~$\tilde\phi$. Indeed, it is a $\bfk$-algebra homomorphism since~$\phi$
  and~$\evl_\ogalg$ are; one sees immediately that it respects the derivation. Let us now check
  that~$\tilde\phi$ also respects the Rota-Baxter structure,
  meaning~$\tilde\phi \big(\cum (f \otimes f_\epsilon)\big) = \cum_\ogalg \, \tilde\phi (f \otimes
  f_\epsilon)$.
  For the left-hand side, we apply~$\tilde\phi$
  to~$\cum \, f \otimes f_\epsilon = (\vcum f) \otimes f_\epsilon - 1 \otimes (f_\epsilon \vcum f)$
  to obtain
  \begin{align*}
    \phi(\vcum f) \, \evl_\ogalg \big( \phi(f_\epsilon) \big) -
    \evl_\ogalg \big( \phi(f_\epsilon \vcum f) \big) = \evl_\ogalg
    \big( \phi(f_\epsilon) \big) \, \Big( \phi(\vcum f) - \evl_\ogalg
    \big( \phi( \vcum f) \big) \Big)
  \end{align*}
  using the multiplicativity of~$\phi$ and~$\evl_\ogalg$ on the second
  term. Since by definition~$1_\ogalg - \evl_\ogalg = \cum_{\!\ogalg}
  \der_\ogalg$, the parenthesized expression above is~$\cum_{\!\ogalg}
  \, \der_\ogalg \, \phi (\vcum f) = \cum_{\!\ogalg} \, \phi \big(
  \der \vcum f) = \cum_{\!\ogalg} \, \phi(f)$. For the right-hand
  side, using~$\cum_{\!\ogalg}$ on~$\tilde\phi( f \otimes f_\epsilon)
  = \phi(f) \, \evl_\ogalg \smash{\big( \phi(f_\epsilon) \big)}$
  yields~$\evl_\ogalg \smash{\big( \phi(f_\epsilon) \big)} \,
  \cum_{\!\ogalg} \phi(f)$ since~$(\ogalg, \der_\ogalg,
  \cum_{\!\ogalg})$ is an integro-differential algebra
  and~$\cum_{\!\ogalg}$ is linear over~$\ker{\der_\ogalg} =
  \im{\,\evl_\ogalg}$ by \mcite[Rem.~2.6(d)]{GuoRegensburgerRosenkranz2012}.
\end{proof}

The crucial point of the embedding of~$\galg[\der,\vcum]$ into a ring
of integro-differential operators is that~$\cum f \der^k$, though a
normal form of~$\galg[\der,\vcum]$, splits when viewed as an
integro-differential operator. Its reduction to normal forms can be
computed as follows.

\begin{lemma}
  \label{lem:ev-rung}
  Let~$(\galg, \der, \cum)$ be an integro-differential algebra. Then
  we have
  \begin{equation}
    \label{eq:ev-rung}
    \cum f \der^k = \sum_{i=0}^{k-1} (-1)^i \big( f^{(i)} -
    \evl(f^{(i)}) \, \evl \big) \der^{k-i-1} + (-1)^k \cum f^{(k)}
  \end{equation}
  for all~$k>0$.
\end{lemma}

\begin{proof}
  We use induction on~$k>0$. The base case~$k=1$ follows from the
  $\ID$-relator of
  Proposition~\ref{prop:lin-op-rings}\ref{it:intdiffop}. Assume now
  the claim holds for some~$k>0$. Then we have
  \begin{equation*}
    \cum f \der^{k+1} = \sum_{i=0}^{k-1} (-1)^i \big( f^{(i)} -
    \evl(f^{(i)}) \, \evl \big) \der^{k-i} + (-1)^k \cum f^{(k)}
    \der,
  \end{equation*}
  and the last term yields~$(-1)^k f^{(k)} + (-1)^{k+1} \evl(f^{(k)})
  \, \evl + (-1)^{k+1} \cum f^{(k+1)}$ by the
  case~$k=1$. Incorporating the first two summands into the summation,
  one obtains~\eqref{eq:ev-rung} with~$k+1$ in place of~$k$, which
  completes the induction.
\end{proof}

We can now provide the embedding of differential Rota-Baxter operators
into a ring of integro-differential operators with ``generic''
integral. The punch line is that one must pass to the free
integro-differential algebra introduced in
Proposition~\ref{thm:free-intdiff}. Since its Rota-Baxter operator
introduces new integration constants, one may view it as being
initialized at a generic point; this will become clearer in
Section~\ref{sec:intdiff-weyl}.

\begin{theorem}
  \label{thm:gen-isom}
  Let~$(\galg, \der, \vcum)$ be an ordinary differential Rota-Baxter
  algebra, and~$(\tilde\galg, \der, \cum)$ the free
  integro-differential algebra defined in
  Proposition~\ref{thm:free-intdiff}. Then the assignment
  \begin{equation}
    \label{eq:gen-isom}
    f \der^k \mapsto f \der^k,\qquad
    f \vcum \tilde{f} \mapsto f \cum \tilde{f},\qquad
    f \vcum \tilde{f} \der^k \mapsto f \cum \tilde{f} \der^k
  \end{equation}
  defines an algebra monomorphism~$\psi\colon \galg[\der,\vcum] \to
  \tilde\galg[\der, \cum]$.
\end{theorem}

\begin{proof}
  From Proposition~\ref{prop:lin-op-rings} we know
  that~$\galg[\der,\vcum] = \galg[\der] \dotplus \galg[\vcum] \!\setminus\! \galg \dotplus [\evl]$,
  where the three components have normal forms~$f\der^k$, $f \vcum \tilde{f}$
  and~$f \vcum \tilde{f} \der^k$, respectively. Hence the map~$\psi$ is well-defined, and it is
  clearly $\bfk$-linear. We can also describe~$\psi$ in a different but equivalent way: Recall
  that~$\galg * \bfk \langle \der, \vcum \rangle$ is a coproduct in the category of (noncommutative)
  algebras, with canonical injections~$i_1\colon \galg \to \galg * \bfk \langle \der, \vcum \rangle$
  and~$i_2\colon \bfk \langle \der, \vcum \rangle \to \galg * \bfk \langle \der, \vcum \rangle$.
  Similarly, $\tilde\galg * \bfk \langle \der, \cum \rangle$ is a coproduct with canonical
  injections~$\tilde{\i}_1$ and~$\tilde{\i}_2$. Then by the universal property for the
  coproduct~$\galg * \bfk \langle \der, \vcum \rangle$, there is an algebra
  morphism~$j\colon \galg * \bfk \langle \der, \vcum \rangle \to \tilde{\galg} * \bfk \langle \der,
  \cum \rangle$
  such that~$j \circ i_1 = \tilde{\i}_1 \circ \iota$ and~$j \circ i_2 = \tilde{\i}_2 \circ i$,
  where~$i\colon \bfk \langle \der, \vcum \rangle \to \bfk \langle \der, \cum \rangle$ is the
  (trivial) isomorphism that renames~$\vcum$ into~$\cum$.
  Writing~$[\RB] \subset \galg * \bfk \langle \der, \vcum \rangle$
  and~$[\ID] \subset \tilde{\galg} * \bfk \langle \der, \cum \rangle$ for the relator ideals
  of~$\galg[\der,\vcum]$ and~$\tilde\galg[\der,\cum]$, respectively, we
  have~$j \, [\DRB] \subset [\ID]$ and hence~$\tilde{p} j [\DRB] = 0$ for the canonical
  projection~$\tilde{p}\colon \tilde\galg * \bfk \langle \der, \cum \rangle \to
  \tilde\galg[\der,\cum] = \big( \tilde\galg * \bfk \langle \der, \cum \rangle \big) \big/ [\ID]$.
  Writing
  $p\colon \galg * \bfk \langle \der, \vcum \rangle \to \galg[\der,\vcum] = \big( \galg * \bfk
  \langle \der, \vcum \rangle \big) \big/ [\DRB]$
  for the other projection, we conclude
  that~$\tilde{p}j\colon \galg * \bfk \langle \der,\vcum \rangle \to \tilde\galg[\der,\cum]$
  descends to an algebra morphism $\galg[\der,\vcum] \to \tilde\galg[\der,\cum]$, which is easily
  recognized as~$\psi$ so that~$\tilde{p}j = \psi p$.

  % By~\cite[Thm.~2.9]{Hungerford1980} we have~$\tilde{p}j = \psi p$
  % and~$\ker{\psi} = \ker{\tilde{p} j} + [\DRB]$.

  It remains to prove that~$\psi$ is injective. Recall that
  although~$f \vcum \tilde{f} \der^k \in \galg[\der, \vcum]$ is a normal form, this is not the case
  for its image~$f \cum \tilde{f} \der^k \in \tilde\galg[\der, \cum]$. In fact, we will apply
  Lemma~\ref{lem:ev-rung} for rewriting the latter as a $\bfk$-linear combination of
  $\tilde\galg[\der, \cum]$-normal forms. Now to show that~$\psi$ is injective,
  assume~$\psi( \sum_j w_j) = 0$ with~$w_j \neq 0$.
  Since~$\tilde\galg[\der,\cum] = \tilde\galg[\der] \dotplus \tilde\galg[\cum] \!\setminus\!
  \tilde\galg \dotplus (\evl)$,
  those~$\psi(w_j) \in \tilde\galg[\der,\cum]$ in the sum~$\sum_j \psi(w_j) = 0$ that belong
  to~$\galg[\der]$ and~$\galg[\cum] \!\setminus\! \galg$ must cancel with corresponding
  contributions in the expansion~\eqref{eq:ev-rung} of the
  other~$\psi(w_j) \in \tilde\galg[\der,\cum]$. Hence we are left with a sum of the
  form~$\sum_{k,l} w_{k,l} = 0$ of evaluation terms coming from~$\psi(f_l \, \vcum g_l \, \der^k)$,
  which are given by
  \begin{equation*}
    w_{k,l} =
    \sum_{i=0}^{k-1} (-1)^{i+1} f_l \, \evl(g_l^{(i)}) \, \evl \der^{k-i-1}
    = \sum_{i=0}^{k-1} (-1)^{i+1} (f_l \otimes g_l^{(i)}) \, \evl
    \der^{k-i-1} \;\in\; (\evl) \subset \tilde\galg[\der,\cum] .
  \end{equation*}
  Let~$\bar{k}$ be the highest exponent~$k$ occurring among the~$\psi(f_l \, \vcum g_l \, \der^k)$,
  and set~$w_l := w_{l,\bar{k}}$. Since the~$\evl \der^i$ are $\bfk$-linearly independent,
  extracting the highest-order terms~$\evl \der^{\bar{k}-1}$, corresponding to~$i=0$ in the above
  sum, yields the relation $\sum_l f_l \otimes g_l = 0$. Applying the
  criterion~\mcite[Lem.~6.4]{Eisenbud1995} there exist~$a_{lm} \in \bfk$ and~$h_m \in \galg$ such
  that~$\sum_{m} a_{lm} h_m = f_l$ for each~$l$, and~$\sum_{l} a_{lm} \, g_l = 0$ for each~$m$. This
  implies~$\sum_l f_l \, \vcum g_l \, \der^k = \sum_m h_m \, \vcum \big( \sum_l a_{lm} g_l \big)
  \der^k = 0$,
  which means that there are no $k$-th order terms~$w_j \in [\evl] \subset \galg[\der,\vcum]$ in the
  original sum~$\psi(\sum_j w_j) = 0$. By induction on~$k$, we conclude that~$\sum_j w_j$ has in
  fact no term~$w_j \in [\evl]$. But then there are no terms~$w_j \in \galg[\der]$
  or~$w_j \in \galg[\vcum] \setminus \galg$ since their images in~$\smash{\tilde\galg}[\der,\cum]$
  would have nothing to cancel. Hence~$\sum_j w_j = 0$, completing the proof that~$\psi$ is
  injective.
\end{proof}

Since the ring of integro-differential operators is rather
well-understood~\mcite{RosenkranzRegensburger2008a,RosenkranzRegensburgerTecBuchberger2012}, it is
advantageous to have a description of the less familiar and somewhat
more subtle differential Rota-Baxter operator ring~$\galg[\der,\vcum]$
as a subring of~$\tilde\galg[\der,\cum]$. However, there is a price to
pay---the \emph{expansion of integration constants} from~$\bfk \subset
\galg$ to~$\bfk \otimes_\bfk \galg \subset \tilde\galg$. This becomes
even more transparent in the case of polynomial coefficients, which we
describe next.

%%%%%%%%%%%%%%%%%%%%%%%%%%%%%%%%%%%%%%%%%%%%%%%%%%%%%%%%%%%%%%%%%%%%%%%%

% =============================================================================
\section{The Integro-Differential Weyl Algebra}
\label{sec:intdiff-weyl}
% =============================================================================

It is most efficient for our purposes to view the classical \emph{Weyl
  algebra}~$A_1(\bfk)$ in the language of skew polynomial
rings~\mcite[\S7.3]{Cohn2003a} since this allows a smooth passage to
the Rota-Baxter case and moreover provides a convenient framework for
algorithmic tasks~\mcite{ChyzakSalvy1998}. Let us recall the basic
setup (without the twist endomorphism that we shall not need here).
For this section we assume that~$\bfk$ is a field of characteristic zero.

If~$A$ is a $\bfk$-algebra\footnote{In contrast
  to~\mcite[\S7.3]{Cohn2003a}, we do not require that~$A$ be a domain
  though if it is then~$A[\xi; \delta]$ is as well. One sees
  easily~\mcite[\S1.1.2]{McConnellRobson2001} that the construction
  works for any~$\bfk$-algebra~$A$, and this will indeed be crucial
  for our definition of the integro-differential Weyl algebra.}  with
derivation~$\delta\colon A \to A$ and $\xi$ an indeterminate, the
\emph{skew polynomial ring}~$A[\xi; \delta]$ is the free left
$A$-module~$\bigoplus_{n \ge 0} A \xi^n$ with
$\bfk$-basis~$\xi^n$. The multiplication on~$A[\xi; \delta]$ extends
the one on~$A$ through the rule~$\xi a = a \xi + \delta(a)$, subject
to the obvious identifications~$\xi^0 = 1, \xi^{n+1} = \xi
\xi^n$. Extending also the homotheties in the obvious way, one obtains
a $\bfk$-algebra~$A[\xi; \delta]$ that contains~$A$ as subalgebra.

Using~$A = \bfk[x]$ with the standard derivation~$\delta = d/dx$ and
indeterminate~$\xi = \der$ yields the one-dimensional Weyl
algebra~$A_1(\bfk) = \bfk[x][\der; \delta]$. One can also introduce
the $n$-dimensional Weyl algebra in a similar way, starting with the
algebra~$A = \bfk[x_1, \dots, x_n]$ and the derivations~$\delta_k =
\der/\der x_k$ and adjoining the indeterminates~$\xi_k = \der_k$ to
obtain the skew polynomial ring~$A_n(\bfk) = \bfk[x_1, \dots,
x_n][\der_1, \delta_1] \cdots [\der_n, \delta_n]$. Here we restrict
ourselves to the one-dimensional case, for which we shall henceforth
use the alternative notation~$\diffweyl := A_1(\bfk)$ as
in~\mcite{RegensburgerRosenkranzMiddeke2009}. In view of the upcoming
Rota-Baxter analogs, we refer to~$\diffweyl$ as the \emph{differential
  Weyl algebra}.

To be more precise, we actually employ the opposite route of
defining~$\diffweyl := \bfk[\der][x; \delta]$ where~$\delta\colon
\bfk[\der] \to \bfk[\der]$ is now defined as the \emph{negative} of
the standard derivation, so that $\delta(\der^n) := -n \,
\der^{n-1}$. One sees immediately that both definitions are equivalent
since the Weyl algebra enjoys the well-known automorphism~$x
\leftrightarrow -\der$. The reason for this unusual definition is
that, for the Rota-Baxter counterpart of~$\diffweyl$, only the second
definition will work.\footnote{Following the standard definition of
  the differential case, one would need~$\delta(x) = -\ader^2$, which
  does not yield a derivation~$\delta\colon \bfk[x] \to
  \bfk[x]$. Algorithmically speaking, the problem is that while the
  degrees of~$\der$ are decreasing, the ones of~$\ader$ are
  increasing; see the remark
  before~\mcite[Def.~2]{RegensburgerRosenkranzMiddeke2009}.} Indeed,
we introduce~\mcite{RegensburgerRosenkranzMiddeke2009} the
\emph{integro Weyl algebra}~$\intweyl := \bfk[\ader][x; \delta]$ with
the derivation~$\delta(\ell^n) := + n \, \ell^{n+1}$. Note that here
as in~\mcite{RegensburgerRosenkranzMiddeke2009} we use~$\ader$ rather
than~$\cum$ for the Rota-Baxter operator; this improves the
readability of iterated integrals and emphasizes the dual nature
of~$\der$ and~$\ader$.

Both derivations are fully determined by their action on the
generators, namely~$\delta(\der) = -1$ and~$\delta(\ader) =
\ader^2$. The former encodes the Leibniz axiom in the commutator
form~$[x, \der] = 1$, the latter the analogous Rota-Baxter axiom~$[x,
\ader] = \ader^2$. Let us now make this precise by comparing those
Weyl algebras with the corresponding \emph{linear operator rings} of
Section~\ref{sec:oprings-modules}. From now on, all weights are zero;
we shall suppress the subscript~$\lambda=0$ for the standard varieties
in~$\Diff[0]$, $\RB[0]$, $\DRB[0]$ and $\ID[0]$.

\begin{lemma}
  \label{lem:diff-and-int-isom}
  We have~$\diffweyl \cong \bfk[x][\der\Mid\Diffz]$ and~$\intweyl
  \cong \bfk[x][\ader \Mid \RBz]$ as~$\bfk$-algebras.
\end{lemma}

\begin{proof}
  By Proposition~\ref{prop:lin-op-rings} we know
  that~$\bfk[x][\der\Mid\Diffz]$ and~$\bfk[x][\ader\Mid\RBz]$ are
  respectively defined by the Leibniz relation~$\der f = f \der + f'$
  and the Rota-Baxter relation~$\ader f \ader = f\backquote \ader -
  \ader f\backquote$; the latter employs the notation~$\ader$ instead
  of~$\cum$ for uniformity. Clearly, it is enough to require the
  relations on the~$\bfk$-basis~$x^n$ of~$\bfk[x]$. But the Leibniz
  relation for~$f = x^n$ follow immediately by a simple induction
  argument from the special case~$f=x$, which is just the commutator
  relation~$[\der, x] = -1$. For the Rota-Baxter relation, we show now
  that it suffices to take the special case~$f=1$, embodied in the
  commutation~$[x, \ader] = \ader^2$. We use induction on~$n>0$ to
  prove~$n \, \ader x^{n-1} \ader \equiv x^n \ader - \ader x^n$ modulo
  the two-sided ideal~$(x \ader - \ader x - \ader^2)$. The base
  case~$n=1$ being trivial, assume the claim for~$n \ge 1$. Then we
  have
  \begin{align*}
    \ader x^n \ader &\equiv n \, \ader^2 x^{n-1} \ader + \ader^2 x^n
    \equiv n \, (x \ader - \ader x) x^{n-1} \ader + (x \ader - \ader
    x) x^n\\
    \text{or} \qquad
    (n+1) \, \ader x^n \ader &\equiv n \, x \ader x^{n-1} \ader + x
    \ader x^n - \ader x^{n+1} \equiv x^{n+1} \ader - \ader x^{n+1},
  \end{align*}
  where the last step uses the induction hypothesis. Hence we obtain
  \begin{equation}
    \label{eq:diff-and-weyl-isom}
    \bfk[x][\der\Mid\Diffz] \cong
    \bfk\langle x, \der\rangle/(\der x - x \der - 1) \quad\text{and}\quad
    \bfk[x][\ader\Mid\RBz] \cong \bfk\langle x, \ader\rangle/(x \ader -
    \ader x - \ader^2).
  \end{equation}
  Using the reductions~$x \der \to \der x - 1$ and~$x \ader \to \ader
  x + \ader^2$ for the ideals in~\eqref{eq:diff-and-weyl-isom}, this
  corresponds to the multiplication in the skew polynomial
  rings~$\diffweyl = \bfk[\der][x; \delta]$ and~$\intweyl =
  \bfk[\ader][x; \delta]$.
\end{proof}

The integro Weyl algebra shares certain common features with its
differential counterpart but also exhibits some \emph{striking
  differences}~\mcite{RegensburgerRosenkranzMiddeke2009}. While both
are Noetherian integral domains, the differential Weyl algebra is a
simple ring but the integro Weyl algebra is not. On the other hand,
the latter comes with a natural grading whereas the former only enjoys
filtration. In this paper, we are not concerned with their further
study. Let us just mention the following noteworthy commutations.

\begin{lemma}
  \label{lem:x-lj}
  We have the commutations~$[x^i, \ader] = i \, \ader x^{i-1} \ader$
  and~$[x, \ader^j] = j \, \ader^{j+1}$ in~$A_1(\ader)$.
\end{lemma}

\begin{proof}
  The first commutation
  is~\mcite[Lem.~11]{RegensburgerRosenkranzMiddeke2009} and follows
  also from the proof of Lemma~\ref{lem:diff-and-int-isom} above. The
  second commutation is the defining property of~$A_1(\ader) =
  \bfk[\ader][x; \delta]$ as a skew polynomial ring.
\end{proof}

For introducing the \emph{integro-differential Weyl
  algebra}~$\intdiffweyl$ one needs a coefficient
domain~$\bfk[\der,\ader]$ that contains~$\bfk[\der]$ as well
as~$\bfk[\ader]$, subject to the natural requirement~$\der \ader =
1$. In other words, we set~$\bfk[\der,\ader] = \bfk\langle D, L
\rangle / (DL - 1)$ where~$\der$ and~$\ader$ are the residue classes
of~$D$ and~$L$, respectively. This ring has been studied extensively;
see for example~\mcite{Jacobson1950,Gerritzen2000}. The
derivation~$\delta$ on~$\bfk[\der,\ader]$ is determined uniquely as an
extension of the derivations on~$\bfk[\der]$
and~$\bfk[\ader]$. Defining now the integro-differential Weyl algebra
by~$\intdiffweyl := \bfk[\der,\ader][x; \delta]$, it is immediately
clear that~$\intdiffweyl$ contains~$\diffweyl$ and~$\intweyl$ as
subalgebras.

We refer again to~\mcite{RegensburgerRosenkranzMiddeke2009} for some
basic algebraic properties of~$\intdiffweyl$; for deeper and more
general results one may consult~\mcite{Bavula2013} and the references
therein. Let us only mention that~$\intdiffweyl$ is neither Noetherian
nor free of zero divisors. Writing~$\evl := 1 - \ader \der \in
\intdiffweyl$ for what we call again the \emph{evaluation}, we
have~$\intdiffweyl = \diffweyl \dotplus \intweyl \!\setminus\! \bfk[x]
\dotplus (\evl)$ as
\mbox{$\bfk$-modules}~\mcite[(18)]{RegensburgerRosenkranzMiddeke2009}. The
resemblance with the decomposition of Lemma~\ref{lem:drbo-decomp} is
no coincidence as can be seen in Corollary~\ref{cor:intdiff-isom}
below.

\begin{lemma}
  \label{lem:intdiffweyl-bases}
  For~$\intdiffweyl$ one may choose the $\bfk$-bases $\mathfrak{B}_i
  := \mathfrak{D} \cup \mathfrak{R} \cup \mathfrak{E}_i \; (i = 1,
  2, 3)$ containing the subbases~$\mathfrak{D} = \{ x^i \der^k \mid
  i,k \geq 0 \}$ and~$\mathfrak{R} = \{ x^i \ell^j \mid i \geq 0; \, j
  > 0 \}$ together with
  \begin{enumerate}
  \item $\mathfrak{E}_1 := \{ x^i \ader^j \evl \der^k \mid i,j,k \geq 0
    \}$,
  \item\label{it:bas2} $\mathfrak{E}_2 := \{ x^i \ader^j \der^k \mid i
    \geq 0; \, j,k > 0\}$,
  \item $\mathfrak{E}_3 := \{ x^i \ader x^j \der^k \mid i, j \geq 0; \,
    k > 0 \}$.
  \end{enumerate}
  Hence one has~$\mathfrak{B}_2 = \{ x^i \ader^j \der^k \mid i, j, k
  \geq 0 \}$ for the case~\ref{it:bas2}. Moreover, one may also use
  the subbasis~$\mathfrak{R}' := \{ x^i \ader x^j \mid i, j \geq 0 \}$
  in place of~$\mathfrak{R}$.
\end{lemma}

\begin{proof}
  The basis~$\mathfrak{B}_1$ has already been derived; see the
  observation
  before~\mcite[Lem.~19]{RegensburgerRosenkranzMiddeke2009}. Both~$\mathfrak{R}$
  and~$\mathfrak{R}'$ are known to be~$\bfk$-bases of~$\intweyl
  \!\setminus\! \bfk[x] \le \intdiffweyl$, called the right basis and
  the mid basis;
  see~\mcite[Cor.~12]{RegensburgerRosenkranzMiddeke2009} and the
  remark
  before~\mcite[Lem.~11]{RegensburgerRosenkranzMiddeke2009}. Hence the
  subbases~$\mathfrak{R}$ and~$\mathfrak{R}'$ are interchangeable.

  Let us next prove that~$\mathfrak{B}_2$ is also a basis. We
  write~$\mathfrak{D}(n) := \{ x^i \der^k \mid 0\leq i,k\leq n\}$ and
  $\mathfrak{R}(n) := \{ x^i \ader^j \mid 0\leq i\leq n, \, 0 < j\leq
  n \}$ for the truncations of~$\mathfrak{D}$
  and~$\mathfrak{R}$. Likewise we have $\mathfrak{E}_1(n) \!:=\! \{
  x^i \ader^j \evl \der^k \mid 0\leq i\leq n; \, 0\leq j,k < n \}$ and
  $\mathfrak{E}_2(n) \!:=\! \{ x^i \ader^j \der^k \mid 0\leq i\leq n;
  \, 0 < j,k\leq n\}$ for the truncated complements. Now set
  $\mathfrak{B}_i(n) := \mathfrak{D} \cup \mathfrak{R} \cup
  \mathfrak{E}_i$ for~$i=1$ and~$i=2$. Clearly we
  have~$|\mathfrak{B}_1(n)| = |\mathfrak{B}_2(n)|$ and
  \begin{equation*}
    \dirlim \mathfrak{B}_i(n) = \mathfrak{B}_i
  \end{equation*}
  for~$i=1$ and~$i=2$. Since
  \begin{equation*}
    \ader^j \evl\der^k = \ader^j(1 - \ader
    \der)\der^k = \ader^j\der^k - \ader^{j+1} \der^{k+1}
  \end{equation*}
  we see that~$\mathfrak{B}_2(n)$ generates~$\bfk
  \mathfrak{B}_1(n)$. But the latter has~$\mathfrak{B}_1(n)$ for a
  basis since it is a subset of the $\bfk$-basis~$\mathfrak{B}_1$
  of~$\intdiffweyl$. Since~$\mathfrak{B}_2(n)$ is thus a generating
  set of the same cardinality, we conclude that~$\mathfrak{B}_2(n)$ is
  also a $\bfk$-basis of~$\bfk \mathfrak{B}_1(n)$ and hence linearly
  independent, and so are those of~$\mathfrak{B}_2 = \smash{\dirlim}
  \mathfrak{B}_2(n)$. It follows that~$\mathfrak{B}_2$ is
  a~$\bfk$-basis of~$\intdiffweyl$.

  In fact, one can easily exhibit an explicit basis transformation
  between~$\mathfrak{B}_1$ and~$\mathfrak{B}_2$. We define $\psi\colon
  \bfk \mathfrak{B}_2 \to \bfk \mathfrak{B}_1$ by
  fixing~$\mathfrak{D}$ and~$\mathfrak{R}$ while setting
  \begin{equation*}
    \psi(x^i \ader^j \der^k) =
    \begin{cases}
      x^i(\ader^{j-k} - \sum_{m=1}^{k} \ader^{j-m} \evl \der^{k-m }) &
      \text{if $j \geq k$},\\[0.5ex]
      x^i( \der^{k-j} - \sum_{m=1}^j \ader^{j-m} \evl \der^{k-m} ) & \text{if $j < k$}.
    \end{cases}
  \end{equation*}
  Similarly, we define $\phi\colon \bfk \mathfrak{B}_1 \to \bfk
  \mathfrak{B}_2$ by fixing again~~$\mathfrak{D}$ and~$\mathfrak{R}$,
  and sending $x^i \ader^j \evl \der^k$ to $x^i( \ell^j \der^k -
  \ell^{j+1} \der^{k+1} )$. Let us now show that $\psi \circ \phi =1$
  and $\phi \circ \psi =1$. Obviously it suffices now to
  consider~$\mathfrak{E}_1$ and~$\mathfrak{E}_2$. For~$j\geq k$ one
  has
  \begin{align*}
    (\phi \circ \psi ) ( x^i & \ell^j \der^k) = \phi( x^i \ell^{j-k} -
    \sum_{m=1}^{k} x^i \ader^{j-m} \evl \der^{k-m })\\
    &=  x^i \ell^{j-k} -  \sum_{m=1}^{k} x^i ( \ell^{j-m} \der^{k-m} - \ell^{j-m+1} \der^{k-m+1})
    = x^i \ell^j \der^k;
  \end{align*}
  and for~$j< k$ again
  \begin{align*}
    (\phi \circ \psi ) ( x^i & \ader^j \der^k) = \phi( x^i \der^{k-j} -
    \sum_{m=1}^j x^i \ader^{j-m} \evl \der^{k-m} ) \\
    &= x^i \der^{k-j} - \sum_{m=1}^j x^i (\ell^{j-m} \der^{k-m} - \ell^{j-m+1} \der^{k-m+1})
    = x^i \ell^j \der^k.
  \end{align*}
  For the other direction, in case~$j\geq k$ one obtains
  \begin{align*}
    (\psi \circ \phi) (x^i & \ader^j \evl \der^k ) = \psi(  x^i \ell^j
    \der^k - x^i \ell^{j+1} \der^{k+1} )\\
    &=  x^i \Big( \sum_{m=0}^{k} \ader^{j-m} \evl \der^{k-m}  -
    \sum_{m=1}^{k} \ader^{j-m} \evl \der^{k-m} \Big)
    = x^i \ader^j \evl \der^k;
  \end{align*}
  and in case~$j< k$ likewise
  \begin{align*}
    (\psi \circ \phi) (x^i & \ader^j \evl \der^k ) = \psi(  x^i \ell^j \der^k - x^i \ell^{j+1} \der^{k+1}) \\
    &= x^i \Big( \sum_{m=0}^{j} \ader^{j-m} \evl \der^{k-m} -
    \sum_{m=1}^{j} \ader^{j-m} \evl \der^{k-m} \Big) = x^i \ader^j \evl \der^k.
  \end{align*}
  Since~$\mathfrak{B}_1$ is a $\bfk$-basis of~$A_1(\der, \ader)$ we
  have the isomorphism~$A_1(\der, \ell) \cong \bfk \mathfrak{B}_1$,
  which together with the isomorphism~$\phi\colon \bfk \mathfrak{B}_1
  \cong \bfk \mathfrak{B}_2$ yields~$A_1(\der, \ader) \cong \bfk
  \mathfrak{B}_2$, and this implies that $\mathfrak{B}_2$ is also a
  $\bfk$-basis of $A_1(\der, \ell)$ as already proved above.

  For proving that~$\mathfrak{B}_3$ is a $\bfk$-basis of~$A_1(\der,
  \ader)$, one proceeds similarly using the transition
  maps~$\phi\colon \bfk\mathfrak{B}_2 \to \bfk\mathfrak{B}_3$
  and~$\psi\colon \bfk\mathfrak{B}_3 \to \bfk\mathfrak{B}_2$ defined by
  \begin{align*}
    \phi(x^i \ader^j \der^k) &= \sum_{m=0}^{j-1} \frac{(-1)^m}{m! \,
      (j-m-1)!} \, x^{i+j-m-1} \ader x^m \der^k,\\
    \psi(x^i \ader x^j \der^k) &= \sum_{m=0}^j \frac{(-1)^{j-m} \, j!}{m!} \,
    x^{i+m} \ader^{j-m+1} \der^k
  \end{align*}
  in view of the
  identities~\mcite[(17)/(16)]{RegensburgerRosenkranzMiddeke2009}. Alternatively,
  one may use the two truncated bases~$\mathfrak{B}'_2(n) :=
  \mathfrak{D}(n) \cup \mathfrak{R}(n) \cup \mathfrak{E}'_2(n)$
  and~$\mathfrak{B}_3(n) := \mathfrak{D}(n) \cup \mathfrak{R}(n) \cup
  \mathfrak{E}_3(n)$ converging to~$\mathfrak{B}_2 = \dirlim
  \mathfrak{B}'_2(n)$ and~$\mathfrak{B}_3 = \dirlim
  \mathfrak{B}_3(n)$, where one defines
  \begin{align*}
    \mathfrak{E}'_2(n) &:= \{ x^i \ader^j \der^k \mid 0 \le i < n; \,
    0 < j,k \le n; \, i+j \le n \},\\
    \mathfrak{E}_3(n) &:= \{ x^i \ader x^j \der^k \mid 0 \le i,j < n; \,
    0 < k \le n; \, i+j<n \} .
  \end{align*}
  The rest of the argument is then as above, with~$\mathfrak{B}_2'$ in
  place of~$\mathfrak{B}_1$, and~$\mathfrak{B}_3$ in place
  of~$\mathfrak{B}_2$.
\end{proof}

The three bases correspond to direct decompositions~$\intdiffweyl =
\bfk\mathfrak{D} \dotplus \bfk\mathfrak{R} \dotplus \bfk\mathfrak{E}_i
\; (i = 1,2,3)$ with standard components~$\bfk\mathfrak{D} =
A_1(\der)$ and~$\bfk\mathfrak{R} = A_1(\ader) \!\setminus\! \bfk[x]$.
The extra component is either the evaluation ideal~$\bfk\mathfrak{E}_1
= (\evl)$, the left~$\bfk[x]$-submodule~$\bfk\mathfrak{E}_2$, or the
evaluation rung~$[\evl] = \bfk\mathfrak{E}_3$.

\begin{coro}
  \label{cor:intdiff-isom}
  We have~$\intdiffweyl \cong \bfk[x][\der,\vcum \Mid \DRBz]$
  as~$\bfk$-algebras.
\end{coro}

\begin{proof}
  In view of the decomposition in Lemma~\ref{lem:drbo-decomp} and the
  isomorphisms of Lemma~\ref{lem:diff-and-int-isom}, this follows
  immediately from Lemma~\ref{lem:intdiffweyl-bases} since the
  evaluation rung~$[\evl] \leq \bfk[x][\der, \vcum]$ has the
  $\bfk$-basis~$\{ x^i \vcum x^j \der^k \mid i,j \geq 0, \, k > 0 \}$,
  which corresponds to~$\mathfrak{E}_3$.
\end{proof}

It is now easy to derive the following \emph{specialization
  isomorphism}~\mcite[Thm.~20]{RegensburgerRosenkranzMiddeke2009} from
the general quotient result on the differential Rota-Baxter operator
rings.

\begin{prop}
  \label{prop:spec-isom}
  We have~$A_1(\der, \ader) / (\evl x) \cong \bfk[x][\der, \cum \Mid
  \IDz]$ as $\bfk$-algebras.
\end{prop}

\begin{proof}
  Using the isomorphism of Corollary~\ref{cor:intdiff-isom}, this
  follows from Proposition~\ref{prop:lin-op-rings}.
\end{proof}

Note that here we have used the standard Rota-Baxter
operator~$\cum\colon x^k \mapsto x^{k+1}/(k+1)$ for the
integro-differential Weyl algebra and the corresponding
integro-differential operator ring~$\bfk[x][\der, \cum]$. As can be
seen from~\mcite[Thm.~20]{RegensburgerRosenkranzMiddeke2009}, one can
also start from any other integro-differential structure~$(\der,
\cum)$ on~$\bfk[x]$ for obtaining a similar isomorphism except that
one factors out the ideal~$(\evl x - c \evl)$ where~$c := \evl(x) \in
\bfk$ is the \emph{integration constant} associated with the integral
operator~$\cum$.

The specialization isomorphism (Proposition~\ref{prop:spec-isom}) can
be interpreted as ``simulating'' integro-differential operators by
differential Rota-Baxter operators (in the important case of
polynomial coefficients). Since the structure of the latter is finer,
this is in principle not surprising. However, we can also derive a
corresponding \emph{generalization isomorphism} that identifies the
finer ring of differential Rota-Baxter operators as a subalgebra in an
overarching integro-differential operator ring. To this end, we take
our earlier result of the general theory (Theorem~\ref{thm:gen-isom})
and interpret it in the more concrete polynomial setting.

\begin{theorem}
  \label{thm:gen-isom-weyl}
  For~$\epsilon$ transcendental over~$\bfk$, endow~$\tilde{\bfk}[x] =
  \bfk[x,\epsilon]$ with derivation~$\der = \der/\der x$ and
  integral~$\cum = \cum_{\!\epsilon}^x$. Then there is a unique
  $\bfk$-algebra monomorphism
  \begin{equation*}
    \iota\colon \intdiffweyl \hookrightarrow \tilde{\bfk}[x][\der,
    \cum]
  \end{equation*}
  that sends~$\ader$ to~$\cum$ while fixing $x$ and~$\der$.
\end{theorem}

\begin{proof}
  Applying Theorem~\ref{thm:gen-isom} to~$\galg := \bfk[x]$ we observe
  that~$\tilde{\galg} = \bfk[x] \otimes_\bfk \bfk[x] \cong
  \bfk[x,\epsilon]$, defined by
  Proposition~\ref{prop:free-intdiffalg}, has the derivation and
  integral as described in the current theorem. Indeed, $\der(x^i
  \otimes x^j) = \der(x_i) \otimes x^j$ means~$\der(x^i \epsilon^j) =
  (\der/\der x) \, x^i \epsilon^j$ for the derivation while~$\cum (x^i
  \otimes x^j) = (\vcum x^i) \otimes x^j - 1 \otimes (x^j \vcum x^i)$,
  where $\vcum$ denotes the standard Rota-Baxter operator
  on~$\bfk[x]$, translates to
  \begin{equation*}
    \cum x^i \epsilon^j = \tfrac{x^{i+1}}{i+1} \, \epsilon^j -
    \epsilon^j \, \tfrac{\epsilon^{i+1}}{i+1} = \cum_{\!\epsilon}^x \, x^i
    \epsilon^j
  \end{equation*}
  for the integral. Hence~$\tilde{\galg}$ coincides
  with~$\tilde{\bfk}[x]$ as an integro-differential algebra.

  Let us now consider the map~$\iota\colon \intdiffweyl \to
  \tilde{\bfk}[x][\der, \cum]$. Since~$x, \der$ and~$\ader$
  generate~$\intdiffweyl$, the uniqueness claim follows. But we know
  from Corollary~\ref{cor:intdiff-isom} that~$\intdiffweyl \cong
  \galg[\der,\vcum]$, and with this identification the map~$\iota$ is
  clearly the same as the $\bfk$-algebra monomorphism given in
  Theorem~\ref{thm:gen-isom}.
\end{proof}

The intuitive idea behind the generalization isomorphism is that one
adjoins a \emph{generic initialization point}~$\epsilon$ for the
integral~$\cum$. The associated (multiplicative) evaluation~$\evl = 1
- \cum \der$ sends~$f(x,\epsilon) \in \smash{\tilde{\bfk}[x]}$
to~$f(\epsilon,\epsilon) \in \smash{\tilde{\bfk} :=
  \bfk[\epsilon]}$. This yields an isomorphic copy~$\smash{\iota
  \big(\intdiffweyl \big)}$ of the integro-differential Weyl algebra
in~$\smash{\tilde{\bfk}[x][\der,\cum]}$. However, one should observe
that~$\smash{\iota \big(\intdiffweyl \big)}$ by itself is only a
differential Rota-Baxter operator ring and \emph{not} an
integro-differential operator ring: The evaluation~$f(x,\epsilon)
\mapsto f(\epsilon, \epsilon)$ does not restrict to a map on its
coefficient domain~$\bfk[x]$.

One may also derive a $\bfk$-\emph{basis} of~$\smash{\iota
  \big(\intdiffweyl \big)} \le
\smash{\tilde{\bfk}[x][\der,\cum]}$. For any integro-differential
operator ring one has the relation~$\cum f \evl = (\cum f) \, \evl$;
see~\mcite[Table~1]{RosenkranzRegensburger2008a} and
Footnote~\ref{fn:no-eval-rules} in the proof of
Proposition~\ref{prop:lin-op-rings}. Setting~$f=1$ and iterating~$j$
times the integral~$\cum = \cum_\epsilon^x$ one obtains the
relation~$\cum \cdots \cum \evl = (x-\epsilon)^j/j! \, \evl$. Hence
$\iota$ maps the basis elements~$x^i \ader^j \evl \der^k \in
\mathfrak{E}_1$ of Lemma~\ref{lem:intdiffweyl-bases} to~$(1/j!)  \,
x^i(x - \epsilon)^j \, \evl \der^k$ while ``fixing'' those
of~$\mathfrak{D}$ and~$\mathfrak{R}'$.

\bigskip

\noindent {\bf Acknowledgments}: This work was supported by the National Natural Science Foundation
of China (Grant No.\@ 11371177 and 11371178), Fundamental Research Funds for the Central
Universities (Grant No.\@ lzujbky-2017-162), the National Science Foundation of US (Grant No.\@
DMS~1001855), the Engineering and Physical Sciences Research Council of UK (Grant No.\@
EP/I037474/1), and the Austrian Science Fund (FWF Grant No.\@ P30052).

\medskip

\noindent We thank the anonymous referee for valuable suggestions helping to improve the paper.

% \bibliographystyle{myplain}
% \bibliography{DiffRBOperators}

\end{document}